\documentclass[11pt,a4paper,lualatex]{amsart}

\usepackage[marginparwidth=0pt,margin=20truemm]{geometry} % margin

\usepackage{mypackage} % My packages.
\usepackage{mycommand} 
\usepackage{amsmath}% My commands.
\usetikzlibrary{knots}
\usetikzlibrary{positioning}

\newcommand{\el}{I_{\nu}(\ell)}

\renewcommand{\iu}{i}

\allowdisplaybreaks[1] % allow page break in align environment

\usepackage[pdfencoding=auto,hypertexnames=false]{hyperref}
\hypersetup{colorlinks=false}

\numberwithin{equation}{section} % Setting of equation numbers 
\usepackage{mytheoremeng} % My theorem environments (English) with cleveref package

\begin{document}
% --------------------------------------------------------------------------

\title[A proof of The Radial limit conjecture for CGP invariants]{A proof of The Radial Limit Conjecture for Costantino--Geer--Patureau-Mirand Quantum invariants}

\author[W. Misteg\aa{}rd]{William  Elb\ae{}k  Misteg\aa{}rd}
\address{Center for Quantum Mathematics, University	of Southern Denmark, 5230 Odense, Denmark}
\email{wem@imada.sdu.dk}

\author[Y. Murakami]{Yuya Murakami}
\address{Faculty of Mathematics, Kyushu University, 744, Motooka, Nishi-ku, Fukuoka, 819-0395, Japan}
\email{murakami.yuya.896@m.kyushu-u.ac.jp}

\date{}

\maketitle

% --------------------------------------------------------------------------

\begin{abstract}
%We prove the radial limit conjecture for Costantino--Geer--Patureau-Mirand (CGPM) quantum invariants of a general negative definite plumbed three-manifold. 
For a negative definite plumbed three-manifold, we give an integral representation of the appropriate average of the GPPV invariants of Gukov--Pei--Putrov--Vafa, which implies that this average admits a resurgent asymptotic expansion, the leading term of which is the Costantino--Geer--Patureau-Mirand invariant of the three-manifold. This proves a conjecture of Costantino--Gukov--Putrov.  %The integral representation  opens up for future study of quantum modularity and resurgence properties.
\end{abstract}

% --------------------------------------------------------------------------

% --------------------------------------------------------------------------

\section{Introduction} \label{sec:main}

% --------------------------------------------------------------------------

Let $(M,\omega)$ be a pair consisting of an oriented negative definite plumbed $3$-manifold $M$ with a plumbing graph $\Gamma$ (assumed to be a tree with negative definite adjacency matrix) and a cohomology class $\omega \in H^1(M, \Q/2\Z) \smallsetminus H^1(M, \Z/2\Z)$. Let $r\geq 2$ be a positive integer which is not divisible by $4$. Let $\mathcal{C}$ be the non-semisimple ribbon tensor category associated with the so-called ``unrolled'' quantum group $U_{\xi}^H(\mathfrak{sl}(2,\bbC)), \xi=e^{\pi i/r},$ which was constructed in \cite{CGP14}.  Denote  the topological invariant of $(M,\omega)$ defined in \cite{CGP14} using the non-semisimple ribbon tensor category $\mathcal{C}$ by $ Z_r(M,\omega) \in \bbC.$

Let $\Delta \in \mathbb{Q}$ be a certain rational defined below. Set $\mathbb{H}\coloneqq\{z \in \bbC: \Im(z)>0\}$. For each $z \in \bbC$ define $\bm{e}(z)=\exp(2\pi i z)$ and for $\tau \in \mathbb{H}$ set $q=\bm{e}(\tau).$ Let $\fraks \in \mathrm{Spin}^c(M)$. The GPPV $q$-series is the convergent integral power series invariant introduced in \cite{GPPV,GM}, which we denote by
\begin{equation} \label{eq:Zhat} \widehat{Z}(M,\fraks;q) \in q^{\Delta}  \mathbb{Z}[[q]].
\end{equation} We now recall the radial limit conjecture in the non-semisimple setup (\cite[Conjecture 1.2]{Costantino-Gukov-Putrov}). Let $\Delta \in \Q$ be the constant defined in \cref{def:GPPV}. For each $\fraks \in \mathrm{Spin}^c(M)$ let $z_r(\omega,\fraks) \in \bbC$ be the constant defined in \cref{def:constans} below (and denoted by $c^{\mathrm{CGP}}_{\omega,\fraks}$ in \cite{Costantino-Gukov-Putrov}). By a sector, we mean a subset of $\bbC^{\times}$ of the form $\{z=r\bm{e}(\theta): r \in (0,\infty), \theta \in [a,b]\}$, for some $a,b \in \R$. For $\tau \in \mathbb{H}$  define
\begin{equation} 
\label{def:Z_tau}
\widehat{Z}_r(M,\omega;\tau)\coloneqq\bm{e}(-\Delta\tau)\sum_{\mathfrak{s} \in \mathrm{Spin}^c(M)} z_r(\omega,\mathfrak{s})  \widehat{Z} (M,\mathfrak{s};\bm{e}(\tau+1/r)).		
\end{equation} 
%With this notation, the radial limit conjecture can be stated as follows.
\begin{conj}[{\cite[Conjecture 1.2]{Costantino-Gukov-Putrov}}] \label{conj:Costantino-Gukov-Putrov}
For any sector $S \subset \mathbb{H}$ the following limit holds as $\tau\in S$ tends to $0$
	\begin{equation} \label{eq:analyticcontinuation}
		\lim_{\tau \rightarrow 0} \widehat{Z}_r(M,\omega;\tau)=Z_r(M,\omega).
		\end{equation} 
\end{conj} Equation \eqref{eq:analyticcontinuation} can be seen as given an analytic continuation of $Z_r(M,\omega)$ as a function of $\bm{e}(1/r)$ to the interior of the unit disc (see \cref{rem:analyticcontinuation}).  In \cite{Costantino-Gukov-Putrov} the conjecture is discussed in greater generality. In the works \cite{GPPV,GM} a similar conjecture was posed for Witten--Reshetikhin--Turaev quantum  invariants \cite{Witten,RT90, Reshetikhin-Turaev}, and this conjecture has been thoroughly studied \cite{Andersen-Mistegard,FIMT,MM,M_plumbed,M_indefinite,M_GPPV,M-Tarashima,wheeler2024quantum}, also in connection to resurgence \cite{Ecalle81,Ecalle81b,CS16} and quantum modularity \cite{AHLMSS,Zagier_quantum,han2022}. The connection to resurgence and quantum modularity was first presented in the influential work \cite{GMP2016resurgence}, which in turn used results from \cite{LZ,LR}. The theory of resurgence, exponential integrals and Pham--Picard--Lefshetz theory \cite{BerryHowls,DH02,Pham83} has played a central role in quantum topology, in particular in connection to complexification and quantum field theory \cite{DU12,BDSSU17,KY22,GW09} and quantum complex Chern--Simons theory \cite{Andersen-Petersen,AM24,GGM21,GGM23,G08,Garoufalidis-Zagier, Witten11b}. We prove
\begin{thm} \label{thm:main}
	\cref{conj:Costantino-Gukov-Putrov} is true for $(M,\omega)$.
\end{thm}
Our proof of \cref{thm:main} is based on an integral representation of \eqref{def:Z_tau} stated in \cref{thm:integral} below, which is proved using  Gaussian reciprocity  (inspired by \cite{M_GPPV,Costantino-Gukov-Putrov}). This integral representation is inspired by results appearing in \cite{Andersen-Mistegard,GMP2016resurgence}, and via the method of steepest descent \cite{Debye,Fedoryuk} it entails a full resurgent expansion of $\widehat{Z}_r(M,\omega;\tau)$ as $\tau$ tends to $0$, denoted by $\widehat{Z}_r^{\text{Pert}}(M,\omega;\tau)$ (and defined in \eqref{def:BorelandPertub}),  with constant term equal to $Z_r(M,\omega)$.  The integral representation \eqref{eq:integraprep} opens up for future demonstration of quantum modularity.  

\subsection{Integral Representation} \label{sec:integral} Let $V$ denote the set of vertices of $\Gamma$. Let $L \subset S^3$ be the associated surgery link of $M$ (see \cref{sec:CGPM}). For each $v\in V$, let $m_v \in H_1(M)$ be the cycle represented by a meridian of the component of $L$ indexed by $v$. Define $\tilde{\omega}\in (\R/2\Z)^V$ by $\tilde{\omega}_v=\omega(m_v)$ for each $v \in V$.  We can and will assume that for all $v \in V$ we have $\tilde{\omega}_v \notin \Z / 2\Z$, as shown in \cite{CGP14}. %Set $\delta_r= -\Im(2\pi\tilde{\omega}/r)$.
 Let $B$ be the adjacency matrix of $\Gamma$.  Let $Q$ denote the quadratic form on $\bbC^{V}$ associated with $-B$. Write $B^{-1}=(b^{v,w})_{v,w \in V}$ for the inverse matrix.  For each $v \in V$, define $F_{v} \in \mathcal{M}(\bbC)$ by $F_{v}(x)=(x-1/x)^{2-\deg(v)}$. Let $e=(1,...,1)\in \Z^V$. %For each $\nu \in \mu_2^V=\{\pm 1\}^V$, define $\lambda_{\mu}=\lambda \prod_{v \in V} \nu_v^{\deg(v)} \in \bbC^{\times},$ where 
 Let $\lambda\in \bbC^{\times}$ be the constant defined in equation \eqref{def:lambda} below, and define  $G_{\omega,r} \in \mathcal{M}(\bbC^V)$ and  $(\widehat{Z}_{\omega,r,l})_{l=0}^{\infty} \subset \bbC$ by the formulae
\begin{align} 
\label{def:G} G_{\omega,r}(x) \coloneqq&\lambda\sum_{\alpha \in \frac{1}{2}(\tilde{\omega} + r e) + \Z^V / r \Z^V}
\bm{e}\left( \frac{- Q(\alpha)}{r} \right)\prod_{v \in V} F_{v}\left(\bm{e}\left(\frac{\alpha_v}{r} +\frac{x_v}{2\pi i}\right)\right),
\\\label{def:Zhatpertubativecoef}  \widehat{Z}_{\omega,r,l}\coloneqq & \frac{1}{4^l l!}\left(\sum_{v,w \in V} b^{v,w} \frac{\partial}{\partial x_v} \frac{\partial}{\partial x_w}\right)^l(G_{\omega,r})(0), \quad \forall l \in \Z_{\geq 0}.
 \end{align}
The quantity $ \widehat{Z}_{\omega,r,l} $ is well-defined since the function $G_{\omega,r}$ is holomorphic at the origin. 
In fact, it is well-known that $G_{\omega,r}(0)= Z_r(M,\omega)$ (see \cite[Section 3.1]{Costantino-Gukov-Putrov} and \cref{prop:G(1)=Z} below) and consequently $\widehat{Z}_{\omega,r,0}=Z_r(M,\omega)$. %For $\nu\in \{\pm 1\}$, let $\R_{\nu}=\R e^{i\nu \varepsilon}$ where $\varepsilon$ is a small positive parameter, and set  $\Upsilon=(2)^{-\abs{V}}\sum_{\nu \in \{\pm 1\}^V} \bigtimes_{v\in V} \R_{v,\nu_v}$.  
The poledivisor of $x \mapsto G_{\omega,r}(ix)$ is given by 
\begin{equation} \label{def:P}
	\mathcal{P}_{\omega,r}
	= \bigcup_{\substack{
			v \in V:  \deg(v) \geq 3,\\
			\alpha \in \frac{1}{2}(\tilde{\omega} + r e) + \Z^m / r \Z^m 
		}}
	\left\{x \in \bbC^V: \frac{2\pi i\alpha_{v}}{r}+ix_{v} \in \pi i \Z\right\}.
\end{equation}
Let $\varepsilon>0$ be a small positive parameter. For each $\nu \in \mu_2^V=\{-1,1\}^V$,  define the integration contour $\Upsilon_{\nu} \coloneqq i\varepsilon \nu+\R^V$. Note that $\Upsilon_{\nu}\cap \mathcal{P}_{\omega,r}=\emptyset$. Observe that the only stationary point of $Q \colon \bbC^V \rightarrow \bbC$ is the origin. We refer to \cref{sec:Asymptotics} for a presentation of basic concepts from asymptotic analysis.  %Define the integral operator $\mathcal{L}_{\tau}$ by the formula \eqref{def:Laplace}, where $u$ is a suitable smooth test function defined in a sector containing a neighbourhood of $\R _+\subset \bbC$,  and where  $v.p. \int_0^{\infty} $ denotes the principal value, which is presented in detail in \cref{sec:Proofs}\begin{equation} \label{def:Laplace}   \mathcal{L}_{\tau}(u)\coloneqq v.p. \int_0^{\infty} \exp\left(\frac{z}{2\pi i\tau }\right) \frac{u(z) dz}{\sqrt{2\pi i \tau}^{\abs{V}}}.\end{equation} 
% and 

 % For all $\tau \in \mathbb{H}$, set $\tau^*\coloneqq-\tau^{-1}$. Let $\sigma$ be the signature of the adjacency matrix of $\Gamma$. %Finally, let $\Upsilon \subset \bbC$ be the contour... 

% In this article, we use Gaussian reciprocity \cite{DT} and an asymptotic formulae generalizing Murakmi.. to prove 

\begin{thm} \label{thm:integral} The integral representation \eqref{eq:integraprep} holds, and for any sector $S\subset \mathbb{H}$ an application of the method of steepest descent \cite{Fedoryuk} 
	to the right hand side of \eqref{eq:integraprep} gives the asymptotic expansion \eqref{eq:expansion} as $\tau \in S$ tends to $0$
	\begin{align}  \label{eq:integraprep}	\widehat{Z}_{r}(M,\omega;\tau) & =  \sum_{\nu \in \mu_2^V} \left(\frac{\det(B)}{(8\pi^2 i\tau)^{\abs{V}}}\right)^{\frac{1}{2}}\int_{\Upsilon_{\nu}} \exp\left( \frac{Q(x)}{2\pi i \tau}\right) G_{\omega,r}(ix)dx,	 
	% &=v.p.\int_0^{\infty} \exp\left(\frac{z}{2\pi i\tau }\right) \frac{\mathcal{B}_{\omega,r}(z)}{\sqrt{2\pi i \tau}^{\abs{V}}} d z
	\\  \label{eq:expansion}\widehat{Z}_{r}(M,\omega;\tau) &\sim Z_r(M,\omega)+\sum_{l=1}^{\infty} \widehat{Z}_{\omega,r,l} (2\pi i\tau)^{l}.\end{align}
\end{thm}
For small $\delta>0$, let $D_{\delta}(0)\subset \bbC^V\cong \R^V+i\R^V$ be the $2\abs{V}$-dimensional ball centered at the origin and of radius $\delta$. The application of the method of steepest descent proceeds by deforming for each $\nu \in \mu_2^V$ the part of $\Upsilon_{\nu}$ close to the origin to obtain a new contour $\Upsilon'_{\nu}$, such that $\Upsilon'_{\nu}\cap D_{\delta}(0)=\R^V \cap D_{\delta}(0) $, and then applying stationary phase approximation ({\cite[Chapter 7, Lemma 7.7.3]{Hormander83}}) to the multi-dimensional contour integral $\int_{\Upsilon'_{\nu}} \exp\left( \frac{Q(x)}{2\pi i \tau}\right) G_{\omega,r}(ix)dx$.

 It follows from the expansion \eqref{eq:expansion}, that the sequence $(\widehat{Z}_{\omega,r,l})_{l=0}^{\infty}$ is a topological invariant of $(M,\omega)$.
%\begin{equation} \label{Boreltransform} \mathcal{B}_{\abs{V}}\left(\sum_{l=0}^{\infty} \widehat{Z}_{\omega,r,l} (2\pi i\tau)^{l}\right)(z)\coloneqqz^{\frac{\abs{V}}{2}-1}\sum_{l=0}^{\infty} \frac{\widehat{Z}_{\omega,r,l}}{\Gamma((\abs{V}+l)/2))}   z^{l}=B_{\omega,r}(z),\end{equation}where  the right hand side of \eqref{Boreltransform} is the germ at the origin. 
Therefore, we are justified in introducing the following notation
\begin{equation} \label{def:BorelandPertub}\widehat{Z}_{r}^{\mathrm{Pert}}(M,\omega;\tau)\coloneqq\sum_{l=0}^{\infty} \widehat{Z}_{\omega,r,l} (2\pi i\tau)^{l}.
\end{equation}  Using standard theory \cite[Chapter 7]{Arnold88} we can compute the Borel transform (see \cref{def:Borel})) of \eqref{def:BorelandPertub}. Denote by $\Omega\in \Gamma^{\infty}(\wedge^{\abs{V}-1}T^* \R^{V} \setminus \{0\})$ the unique smooth differential form on $\R^V \setminus\{0\}$ of rank $\abs{V}-1$ with the property that $\Omega \wedge dQ=\sqrt{\det(B/\pi)}\wedge_{v \in V} dx_v$. This is a normalization of the Gelfand--Leray form \cite{Leray}.  %We give an explicit formula in equation \eqref{def:GelfandLeray}. Let $z$ be a complex variable and define the meromorphic multivalued function\begin{align} \label{def:BorelTransform}\mathcal{B}_{\omega,r}(z)
%\\\label{def:hatZ_l} \widehat{Z}_{\omega,r,l} \coloneqq&b_{\omega,r,l}  \Gamma\left((\abs{V}+l)/2\right),\end{align}

\begin{cor} \label{cor:Borel} The  Borel transform of the product of \eqref{def:BorelandPertub}  and $(2\pi i \tau)^{\abs{V}/2} $ is given as follows	\begin{equation}  \label{eq:integraprep2}	\mathcal{B}\left( (2\pi i \tau)^{\abs{V}/2}  \widehat{Z}^{\mathrm{Pert}}_{r}(M,\omega;\tau)\right)(z)= z^{\frac{\abs{V}}{2}-1} \int_{Q^{-1}(1)\cap \R^V}  G_{\omega,r}(i  x \sqrt{z} )  \Omega(x). \end{equation} \end{cor}
%The normalization factor $(2\pi i \tau)^{\abs{V}/2}$ in \eqref{def:BorelandPertub} is motivated by Corollary \cref{cor:Borel}. 
 We can and will assume that $\Gamma$ is choosen such that $\abs{V}$ is minimal plumbing graph trees for $M$. Thus we are justified in calling the function on the right hand side of \eqref{eq:integraprep2} the Borel transform of  $\widehat{Z}^{\mathrm{Pert}}_{r}(M,\omega)$, and in defining
\begin{equation} \label{def:BorelM}
\mathcal{B}_r(M,\omega;z)\coloneqq z^{\frac{\abs{V}}{2}-1} \int_{Q^{-1}(1)\cap \R^V}  G_{\omega,r}(i  x \sqrt{z} )  \Omega(x).
\end{equation}
Our study of the Borel transform is motivated by the case of Witten--Reshetikhin--Turaev quantum invariants \cite{RT90,Reshetikhin-Turaev}, where the Borel transform play an important role in connection to ideas motivated by the quantum Chern--Simons field theory picture \cite{Witten} of the Witten--Reshetikhin--Turaev TQFT, as shown in \cite{Andersen-Mistegard,GMP2016resurgence}. 

  Recently, there has been work \cite{creutzig2021} on developing a corresponding quantum field theory for non-semisimple TQFT \cite{CGP14}, and potentially the Borel transform will be of importance in this context also. Along these lines it would be interesting to understand further the significance of $\mathcal{B}_r(M,\omega)$ and the significance of the set of singularities of $\mathcal{B}_r(M,\omega)$.

We make the following remark on \cref{conj:Costantino-Gukov-Putrov}, which also clarify how it relates to results obtained in the work \cite{Costantino-Gukov-Putrov}, which is among the main inspirations for this article.

\begin{rem} \label{rem:analyticcontinuation} We make the following remarks. \begin{itemize} \item Let $D\subset \bbC$ be the unit disc. Our \cref{conj:Costantino-Gukov-Putrov} is a reformulation of \cite[Conjecture 1.2]{Costantino-Gukov-Putrov}, which is formulated with respect to the variable $q=\bm{e}(\tau)\in D$, i.e. they conjecture that the following non-tangential limit holds (with the obvious notational modification)
	\begin{equation} \label{eq:non-tangential}
	\lim_{q \rightarrow \bm{e}(1/r)} \widehat{Z}_r(M,\omega;q)=Z_r(M,\omega).
	\end{equation} \item One is naturally lead to consider limits where $q$ tends to $\bm{e}(1/r)$ along a ray emanating from the origin, and this is the reason for calling \cite[Conjecture 1.2]{Costantino-Gukov-Putrov} the radial-limit conjecture. \item In contrast to \cite[Conjecture 1.2]{Costantino-Gukov-Putrov}, our conjecture is formulated with respect to the variable $\tau \in \mathbb{H}$, and our condition that the limit is taken with respect to an arbitrary closed sector reflects the fact that the limit \eqref{eq:non-tangential} is taken in the non-tangential sense.
	\item In \cite{Costantino-Gukov-Putrov}, \cref{conj:Costantino-Gukov-Putrov} is proven for a certain class of $Y$-shaped plumbing graphs for which the adjacency matrix satisfy an open condition (see {\cite[Example 4.27]{Costantino-Gukov-Putrov}}). For this class,  our \cref{thm:integral} is a strengthening, which provides additional information.
	\end{itemize} 
	\end{rem}

 \subsection{Organization}

In \cref{sec:pre} we recall the definition of the relevant topological invariants. In \cref{sec:Asymptotics}, we recall basic definitions from asymptotic analysis. In \cref{sec:Proofs} we present our proofs. In \cref{sec:generating} we make a key step with \cref{prop:GeneratingFunction}, which makes precise the notion that $G_{\omega,r}$ is in a suitably sense the generating function for $\widehat{Z}_r(M,\omega)$, and which is proven using Gaussian reciprocity (\cref{prop:reciprocity}) and following and refining ideas from \cite{Costantino-Gukov-Putrov}. In Section \cref{sec:integral} we use \cref{prop:GeneratingFunction}, a well-known exact formula for Gaussian integrals and the method of steepest descent \cite{Debye} to prove \cref{thm:integral}. \cref{thm:main} is a direct consequence of \cref{thm:integral}. As a consequence of \cref{thm:main}, it holds in particular that \begin{equation} \label{eq:vertical}
\lim_{t \rightarrow +0} \widehat{Z}_r(M,\omega;it)=Z_r(M,\omega),
\end{equation} and in
 \cref{sec:AlternativeProof}, we give a an alternative proof of \eqref{eq:vertical} which is based on an application of the Euler--Maclaurin type asymptotic expansion stated in \cref{prop:asymp_lim} to the central identity of \cref{prop:GeneratingFunction}. The two approaches are complimentary (see \cref{rem:comparison}). %The merits of the proof based on \cref{thm:integral} are that it puts our work in the context of exponential integrals and resurgence, and enables the explicit computation of the Borel transform given in \cref{cor:Borel}.

\subsection*{Acknowledgements.} 

This work is supported by the grant from the Simons foundation, Simons Collaboration on New Structures in Low-Dimensional Topology grant no. 994320, the  ERC-SyG project, Recursive and Exact New Quantum Theory (ReNewQuantum) with funding from the European Research Council under the European Union's Horizon 2020 research and innovation programme, grant agreement no.~810573 and the Carlsberg Foundation grant no. CF20-0431.
The first author would like to thank Masanobu Kaneko and Yuji Terashima for giving many helpful comments.
The first author is supported by JSPS KAKENHI Grant Number JP23KJ1675.

% --------------------------------------------------------------------------

\section{Quantum Invariants} \label{sec:pre}

% --------------------------------------------------------------------------

% --------------------------------------------------------------------------

\subsection{The GPPV Invariant}

% --------------------------------------------------------------------------

As explained in  \cite{Costantino-Gukov-Putrov}, we have an identification $\Spin^c(M)\cong ( \delta + 2\Z^V) / 2B\Z^V$, where $ \delta \coloneqq (\deg v)_{v \in V} $.
We use this fact implicitly below.
\begin{dfn}[{\cite[Equation (A.29)]{GPPV}}] \label{def:GPPV} Set $\Delta=-(3\abs{V} + \tr B)/4$. Let $\fraks \in ( \delta + 2\Z^V) / 2B\Z^V$.  Define $\Theta_{-B, \fraks} (q;x) \in \Z[[q^{1/4}]][x_v^{\pm}, v \in V], F_{\Gamma} \in \mathcal{M}(\bbC^{V})$ and $\widehat{Z}_{\fraks} (M;q) \in q^{\Delta}  \Z[[q]]$  by
	\begin{align} \label{def:F}
	F_{\Gamma}(x) &\coloneqq \prod_{v \in V}(x_v-1/x_v)^{2-\deg(v)}, \\
	\Theta_{-B, \fraks} (q; x) &\coloneqq 
	\sum_{\ell \in \fraks + 2B(\Z^V)} q^{-\ell^T B^{-1} \ell/4} \prod_{v \in V} x_v^{\ell_v}, \\
	\widehat{Z}_{\fraks} (M;q) 
	&\coloneqq
	q^{\Delta} \,
	\mathrm{v.p.} \oint_{\abs{x_v}=1, v \in V}  \Theta_{-B, \fraks}(q; x)  
	F_{\Gamma}(x) \frac{dx_v}{2\pi\iu x_v}.
	\end{align}
%	for 	\begin{equation}	\Theta_{-B, \fraks} (q; x) \coloneqq 	\sum_{\ell \in \fraks + 2W(\Z^V)} q^{-\transpose{\ell} B \ell/4} \prod_{v \in V} x_v^{\ell_v}	\end{equation}	is the theta function defined for complex numbers $ q $ with $ \abs{q} < 1 $ and $ x  = (x_v)_{v \in V} $,	and $ F\left( (x_v)_{v \in V} \right) $ is the rational funtion defined in \cref{subsubsec:coeff_of_F}. 
\end{dfn}

Here we used the notation $\widehat{Z}_{\fraks} (M;q) $ for the series introduced in \eqref{eq:Zhat}. It was proven in \cite{GM} that the series $\widehat{Z}_{\fraks} (M;q) $ is a topological invariant.

% --------------------------------------------------------------------------

\subsection{The CGPM Invariant} \label{sec:CGPM}

% --------------------------------------------------------------------------
 
 Set $X_r\coloneqq \Z\setminus r \Z$. For $\alpha \in \bbC \setminus X_r$ define $d(\alpha)\coloneqq\sin(\pi \alpha/r)/\sin(\pi \alpha)$.
%\begin{equation} d(\alpha)\coloneqq\frac{\xi^{\alpha}-\xi^{-\alpha}}{\xi^{r\alpha}-\xi^{-r\alpha}}.\end{equation}
Let $\mathcal{L}_r$ denote the set of pairs $(T,\alpha)$, where $T$ is a framed and oriented  trivalent graph in $S^3$ and $\alpha$ is a coloring of the edges of $T$ by elements of $\bbC \setminus X_r$. Observe that a framed oriented link in $S^3$ can be seen as a trivalent graph with no vertices. Let $Z_r:\mathcal{L}_r \rightarrow \bbC$ be the invariant axiomatically defined in \cite[Section 2.2]{CGP14}, where it is denoted by $N_r$ (we use a different normalization for $d(\alpha)$, the effect of this change is accounted for in \cite[Remark 2.3]{CGP14}).

 Let $L\subset S^3$ be a framed and oriented surgery link for $M$. 
Let $b_{+}$ and $ b_{-} $ denote the number of positive and negative eigenvalues of the linking matrix of $L$ respectively. 
Denote by $V=\pi_0(L)$. For each component $L_v,v \in V$, let $m_v$ be the meridian of $L_v$, and define the element $\mu \in (\R/ 2 \Z)^V$ by
$\mu_v=\omega([m_v]) \in \R /2\Z$ for all $v \in V$, where $[m_v]$ is the associated homology class of the meridian $m_v.$  
Recall the notation $\xi=\exp(\pi i/r)$. Define 
\[
\Delta_{-} \coloneqq
\xi^{3/2} r^{1/2} \cdot
\begin{cases}
	\iu & \text{ if } r \equiv 1 \bmod 4, \\
	(1 - \iu) & \text{ if } r \equiv 2 \bmod 4, \\
	(-1) & \text{ if } r \equiv 3 \bmod 4
\end{cases}
\]
and let $\Delta_+=\overline{\Delta_{-}}$. 
Let 
\begin{equation} H_r\coloneqq\{-(r-1),-(r-3),...,(r-3),(r-1)\}.
\end{equation}  

\begin{dfn}[\cite{CGP14}] \label{def:Z_r} Define
	\begin{equation} \label{eq:def:Z_r}
	Z_r(M,\omega)\coloneqq \frac{1}{\Delta_+^{b_+} \Delta_{-}^{b_{-}}} \sum_{\alpha  \in H_r^V+\mu} Z_r(L,\alpha) \prod_{v\in V} d(\alpha_v) .
	\end{equation}
	\end{dfn} 
It is proven in \cite{CGP14} that the right hand side of \eqref{eq:def:Z_r}, which a priori depends on the surgery link $L$, is a topological invariant of $(M,\omega)$. We remark that for a surgery link $L$, we can take the one naturally encoded in the plumbing graph $\Gamma$. This link has an unknotted component $L_v$ for each vertex $v\in V$ of the graph, and the framing of $L_v$ is given by the weight of the vertex $v$. Two components $L_v,L_{v'}$ have linking number equal to unity if and only if $v$ and $v'$ are joined by an edge, otherwise $L_v\cup L_{v'}$ is a split disjoint union. For this surgery link, we have that $\mu=\tilde{\omega}$, where $\tilde{\omega} \in \R^V / 2\Z^V$ is the element that was defined in \cref{sec:integral}.

% --------------------------------------------------------------------------

\subsection{Constants}

% --------------------------------------------------------------------------

For $s \in \mathrm{Spin}(M)$ denote by $\mu(M,s)$ the Rokhlin invariant of $(M,s)$, and let $q_s:H_1(M,\Z)\rightarrow \Q /2 \Z$ be the refined quadratic form. Denote by $\mathcal{T}(M,[\omega])$ the Reidemeister torsion (with conventions as in \cite{Costantino-Gukov-Putrov}). Let $\sigma:\ H_1(M,\Z)\times\operatorname{Spin}(M) \rightarrow \operatorname{Spin}^c(M)$ be the canonical map. This is surjective.  
Let $\lk \colon H_1(M,\Z) \otimes H_1(M,\Z) \to \Q / \Z $ be the linking form. 
For $a,b,f \in H_1(M,\Z)$ and $s \in \mathrm{Spin}(M)$ define
\begin{align} \begin{split} \label{def:W}
W_{\mathrm{ev},r}(a,b,s)&= -\frac{r}{4}q_s(a)-\lk(a,b)-\frac{\omega(a)}{2},
\\W_{\pm,r}(a,b,f)&=- \frac{r\mp 1}{4}\lk(a,f,\mp b)- \frac{1}{2}\omega(a)\pm \lk(f,f).
\end{split}
\end{align}

\begin{dfn} \label{def:constans} For $\fraks \in \mathrm{Spin}^c(M)$, pick  $(b,\mathrm{s}) \in  H_1(M,\Z)\times\operatorname{Spin}(M)$ with $ \fraks=\sigma(b,\mathrm{s})$ and define 
	\begin{align} \label{coeffsummary}
	& z_r(\omega,\fraks)\coloneqq	\frac{\mathcal{T}(M, [\omega])}{\abs{H_1(M,\Z)}} 
	\begin{cases}
	\mp e^{\mp \frac{\pi \iu}{2} \mu(M,s)}
	\sum_{a, f \in H_1(M, \Z)}
	e^{2\pi \iu W_{\pm,r}(a,b,f), }
	&\text{ if } r=\pm 1 \mod 4,
	\\
	\! \sqrt{|H_1(M,\Z)|} \sum_{a \in H_1(M, \Z)}
	e^{2\pi \iu W_{\mathrm{ev},r}(a,b,s)},
	&\text{ if } r=2 \mod 4
	%\\ e^{\frac{\pi \iu}{2}\mu(M,s)}	\sum_{a, f \in H_1(M, \Z)}	e^{2\pi \iu \left( -\frac{r+1}{4}\lk(a,a) -\lk(a,f+b)-\frac{1}{2}\omega(a) -\lk(f,f) \right)} &\text{ if } r=3\bmod 4.
	\end{cases}
	\\& \label{def:lambda}
\lambda\coloneqq\Delta_{-}^{-\abs{V}}\exp\left(-\pi i \frac{\text{Tr}(B)(r-1)^2}{2r})\right)\calT (M, [\omega]).
\end{align} 
\end{dfn}

% --------------------------------------------------------------------------

\section{Basic Definitions from Asymptotic Analysis} \label{sec:Asymptotics}

% --------------------------------------------------------------------------

We briefly recall key definitions from asymptotic analysis. We refer the interested reader to \cite{Olver}.

% --------------------------------------------------------------------------

%\subsection{Basic Definitions from Asymptotic Analysis}

% --------------------------------------------------------------------------

\begin{dfn}
	Let $\phi \colon S \rightarrow \bbC$ be a function on a sector $S\subset \bbC$ with a specified branch of the logarithm. A Puiseux series $\tilde{\phi}(t)=\sum_{j=m}^{\infty} \phi_j  t^{j/n}, m \in \Z,n \in \Z_+, (\phi_j)_{j=m}^{\infty} \subset \bbC,$ is a Poincare asymptotic expansion of $\phi(t)$ for $t$ tending to $0$ if for every integer $L \geq m$ it holds that
	\begin{equation}
	\phi(t)=\sum_{j=m}^{L} \phi_j  t^{j/n}+ \mathcal{O}(t^{(L+1)/n}), \forall t \in  S.
	\end{equation}
\end{dfn}

We recall that if a function $\phi$ admits a Poincare asymptotic expansion  $\tilde{\phi}$, then this is in fact uniquely determined by $\phi$, and we write $\phi(t) \sim \tilde{\phi}(t)$ to designate that $\tilde{\phi}$ is the asymptotic expansion of $\phi$.

\begin{dfn} \label{def:Borel} Let $\alpha \in \Q_+$, and let $\tilde{\phi}$ be a Puiseux series of the form $\tilde{\phi}(t)=t^{\alpha}\sum_{j=0}^{\infty} \phi_j  t^{j},  (\phi_j)_{j=m}^{\infty} \subset \bbC$. The Borel transform of $\tilde{\phi}$ is the Puiseux series $\mathcal{B}(\tilde{\phi})$ defined by
	\begin{equation}
	\mathcal{B}(\tilde{\phi})(z)=\sum_{l=0}^{\infty} \frac{\phi_l}{\Gamma(\alpha+l)}   z^{\alpha+l-1}.
	\end{equation}
	We say the series $\tilde{\phi}$ is resurgent if $\mathcal{B}(\tilde{\phi})$ is the germ of a holomorphic function on the universal cover of $\bbC \setminus \Omega(\tilde{\phi})$, where $\Omega(\tilde{\phi}) \subset \bbC$ is a discrete subset containing the origin. %A half-axis emanating from the origin and meeting $\Omega$ is called a Stokes ray of $\mathcal{B}(\tilde{\phi})$.
\end{dfn} 
For an introduction to resurgence, we refer to \cite{CS16}. We recall the following Lemma, which provides motivation for the Borel transform, and which is very well-known, see for instance \cite[Chapter 5]{CS16}.
\begin{lem}\label{lem:resummation} The following holds. \begin{enumerate} \item If a Puiseux series $\tilde{\phi}$ is resurgent, and if $\mathcal{B}(\tilde{\phi}) (z)\in \mathcal{O}(\exp(c_0\lvert z \rvert)) $ for some $c_0 \in (0,\infty)$  on a half-axis $\R_+ e^{i\theta} \subset \bbC \setminus \Omega(\tilde{\phi}), \theta \in [0,2\pi)$, then the expression \eqref{def:resum} defines a holomorphic function on $\{ t \in \bbC: \Re(e^{\pi i\theta}/t)>c_0\}$, which has $\tilde{\phi}$ as Poincare asymptotic expansion as $t$ tends to $0$
\begin{align} \label{def:resum}
\phi_{\theta}(t)\coloneqq&\int_{\R_+ e^{i\theta}} e^{-z/t} \mathcal{B}(\tilde{\phi})(z) dz \sim \tilde{\phi}(t).
\end{align}
\item Conversely, if $B \colon \R_+\rightarrow \bbC$ is an absolutely integrable  function which admits a convergent expansion of the form $B(z)=\sum_{l=0} b_l  z^{\alpha+l-1}$ with $\alpha \in \Q_+$ and $(b_l)_{l=0}^{\infty} \subset \bbC$, then the following Poincare asymptotic expansion holds as $t$ tends to $0$
\begin{equation}
\int_0^{\infty} e^{-z/t} B(z) dz \sim \tilde{\phi}(t)\coloneqq\sum_{l=0}^{\infty} \Gamma(\alpha+l) b_l  t^{\alpha+l},
\end{equation}
and it holds that the Borel transform $\mathcal{B}(\tilde{\phi})$ is equal to the germ of $B$ at $0$.
\end{enumerate}
\end{lem} In the case the first assumption of the above lemma holds, one says that $\phi_{\theta}$ is the Borel--Laplace resummation of $\tilde{\phi}$ of in direction $e^{i\theta}.$ It is in this sense, that the Borel transform is the formal inverse of the Laplace transform. %These remarks explains the validity of the claims made in \cref{rem:Borel}. 
The above concepts are very well illustrated with the example of an exponential integral with holomorphic phase function \cite{BerryHowls,DH02}.

\section{Proofs} \label{sec:Proofs}

% --------------------------------------------------------------------------

%We remark that in this section, we are repeating some elements of the calculation done in \cite[Section 3.1]{Costantino-Gukov-Putrov}. The reason is to make this article as self-contained as possible.

%. The second reason is that, as explicitly remarked in \cite[Section 3.1]{Costantino-Gukov-Putrov}, some of their arguments are somewhat cavalier concerning convergence and limits, and we emphasize that all computations in this article are rigorous.

% --------------------------------------------------------------------------

\subsection{The Generating Function} \label{sec:generating}

% --------------------------------------------------------------------------

\begin{prop} \label{prop:G(1)=Z} We have that $G_{\omega,r}(0)=Z_r(M,\omega).$
	\end{prop}
We remark that the content of \cref{prop:G(1)=Z} is a simple reformulation of a well-known formula for $Z_r(M,\omega)$ stated in \cite[Section 3.1]{Costantino-Gukov-Putrov}. We include it here to make the article self-contained.
\begin{proof}
	 Denote by $L_{\Gamma}$ the surgery link of $M$ obtained from the plumbing graph $\Gamma.$ Let $E$ denote the set of edges of $\Gamma$. For any $\alpha, \beta \in \bbC$ define
	\begin{equation} \label{def:structureconstants}
	\begin{matrix} & S(\alpha,\beta)\coloneqq\xi^{\alpha \beta}, & T(\alpha)\coloneqq\xi^{\frac{\alpha^2-(r-1)^2}{2}}.
	\end{matrix}
	\end{equation} 
	Then we have for all $\alpha \in \mu+H_r^V$ that 
	\begin{equation} \label{Z_rLink}
	Z_r(L_{\Gamma},\alpha)=\prod_{v \in V} d(\alpha_v)^{1-\deg(v)}T(\alpha_v)^{B_{vv}} \prod_{(v,w) \in E} S(\alpha_v,\alpha_w).
	\end{equation}
	This follows from the axioms for $Z_r$ listed in \cite[Section 2.2]{CGP14}. For any $\alpha \in \tilde{\omega}+H_r^V$ we will write $k=\alpha-\tilde{\omega} \in H_r^V$.  From \cref{def:Z_r}, we obtain that
	\begin{align} \begin{split}  \label{eq:Z_r1}
	Z_r(M,\omega)&=\frac{1}{\Delta_+^{b_+}\Delta_{-}^{b_{-}}}\sum_{\alpha \in \tilde{\omega}+H_r^V} \prod_{v \in V} d(\alpha_v)^{2-\deg(v)}T(\alpha_v)^{B_{vv}} \prod_{(v,w) \in E} S(\alpha_v,\alpha_w)
	\\&=\frac{1}{\Delta_+^{b_+}\Delta_{-}^{b_{-}}}\sum_{\alpha \in \tilde{\omega}+H_r^V} \prod_{v \in V} \left(\frac{\xi^{\alpha_v}-\xi^{-\alpha_v}}{\xi^{r\alpha_v}-\xi^{-r\alpha_v}}\right)^{2-\deg(v)}\xi^{B_{v,v}\frac{\alpha_v^2-(r-1)^2}{2}} \prod_{(v,w) \in E} \xi^{\alpha_v \alpha_w}
	\\&=\frac{\xi^{-\frac{\text{Tr}(B)(r-1)^2}{2}}}{\Delta_{-}^{\abs{V}}}\sum_{\alpha \in \tilde{\omega}+H_r^V} (-1)^{\sum_{v \in V} k_v\deg(v)}\prod_{v \in V} \left(\frac{\xi^{\alpha_v}-\xi^{-\alpha_v}}{e^{\pi i\tilde{\omega}_v}-e^{-\pi i \tilde{\omega}_v}}\right)^{2-\deg(v)}\xi^{B_{v,v}\frac{\alpha_v^2}{2}} \prod_{(v,w) \in E} \xi^{\alpha_v \alpha_w},
	\end{split}  
	\end{align}  
	where, in the second equality we simply used the definition \eqref{def:structureconstants}, and in the third equality we used $b_+=0, b_{-}=\abs{V}$. Recall that $B_{v,w}=1$ for every $(v,w) \in E$. Therefore, we have for all $\alpha \in H_r^V+\tilde{\omega}$ that
	\begin{equation} \label{eq:simplifyMatrix}
	\prod_{v \in V} \xi^{B_{v,v}\frac{\alpha_v^2}{2}} \prod_{(v,w) \in E} \xi^{\alpha_v \alpha_w}=\prod_{v \in V} \xi^{B_{v,v}\frac{\alpha_v^2}{2}} \prod_{(v,w) \in E} \xi^{\alpha_v B_{v,w} \alpha_w}=\xi^{\alpha^T B \alpha}.
	\end{equation} 
	Notice that 
	$(-1)^{\sum_{v \in V} k_v\deg(v)}=1$, because either all $k_v$ are even, in which case it is trivial, or all $k_v$ are odd, in which case it follows from the fact that $\sum_{v} \deg(v)$ is even. By using $(-1)^{\sum_{v \in V} k_v\deg(v)}=1$, equation \eqref{eq:simplifyMatrix} and the function $F_{\Gamma}$ introduced in, we can rewrite \eqref{eq:Z_r1} as follows
	\begin{equation} \label{eq:foerste}
	Z_r(M,\omega)=\frac{\xi^{-\frac{\text{Tr}(B)(r-1)^2}{2}}}{\Delta_{-}^{\abs{V}} F_{\Gamma}((\bm{e}(\tilde{\omega}_v/2))_{v\in V})} \sum_{\alpha \in \tilde{\omega}+H_r^V} F_{\Gamma}((\bm{e}(\alpha_v/(2r)))_{v\in V}) \bm{e}\left(\alpha^TB\alpha/(4r) \right).
	\end{equation}
	As explained in \cite[Section 3.1]{Costantino-Gukov-Putrov}, we have that
	$F \left( \left( \bm{e}\left( \tilde{\omega}_v / 2 \right) \right)_{v \in V}  \right)^{-1} = \calT (M, [\omega])$. Thus we have that
	\begin{equation}
	\xi^{-\frac{\text{Tr}(B)(r-1)^2}{2}}(\Delta_{-}^{\abs{V}} F_{\Gamma}((\bm{e}(\tilde{\omega}_v/2))_{v\in V}))^{-1}=\lambda,
	\end{equation} 
	where $\lambda$ is the constant introduced in equation \eqref{def:lambda}. Thus, by inspecting equation \eqref{eq:foerste} and recalling the definition of $G_{\omega,r}$ given in equation \eqref{def:G}, we see that it suffices to show that
\begin{equation}
\sum_{\alpha \in \tilde{\omega}+H_r^V} 
F_{\Gamma}((\bm{e}(\alpha_v/2r))_{v\in V}) 
\bm{e}\left(\alpha^TB\alpha/4r \right)
=
\sum_{\alpha \in \frac{1}{2}(\tilde{\omega} + r \varepsilon) + \Z^V / r \Z^V} 
F_{\Gamma} \left( \left( \bm{e}\left( \frac{\alpha_v}{r} \right) \right)_{v \in V} \right)
\bm{e} \left( \frac{1}{r}\transpose{\alpha}B\alpha \right),
\end{equation} and this follows from the fact for any $\alpha$ and any $ m \in \Z^{V} $ we have that
	\begin{equation}
	(\alpha+2rm)^t B (\alpha+2rm) \equiv \alpha^T B \alpha + 2r \transpose{\tilde{\omega}} B m \bmod 4r\Z.
\end{equation} This finishes the proof. \end{proof}

	%By the above argument, we have	\begin{align}	&\sum_{\alpha \in \mu+H_r^V} F((\bm{e}(\alpha_v/(2r)))_{v\in V}) \bm{e}\left(\alpha^TB\alpha/(4r) \right)=\sum_{n \in \mu + r \varepsilon + 2\Z^V / 2r \Z^V}	\bm{e} \left( \frac{1}{4r}\transpose{n}Bn \right)	F \left( \left( \bm{e}\left( \frac{n_I}{2r} \right) \right)_{v \in V} \right)	\\	&=	\sum_{n \in \frac{1}{2}(\mu + r \varepsilon) + \Z^V / r \Z^V}	\bm{e} \left( \frac{1}{r}\transpose{n}Bn \right)	F \left( \left( \bm{e}\left( \frac{n_I}{r} \right) \right)_{v \in V} \right)	\end{align}

 Recall the notation $Q$ for the quadratic form associated with $-B$. Similarly, we let $Q^{-1}$ denote the quadratic form associated with $-B^{-1}$. Let $\nu \in \mu_2^V=\{1,-1\}^V$. Set $\lambda_{\nu}\coloneqq\lambda  \prod_{v \in V}\nu_v^{\deg(v)}$ and define the diagonal matrix $I_{\nu}\coloneqq\diag((\nu_v)_{v\in V})$ and the quadratic form $Q_{\nu}^{-1}\coloneqq-\transpose{I_{\nu}}B^{-1}I_{\nu}$.  Define the following two sequences
 \begin{align} & \label{def:Fseq} (F^{\nu}_{\ell})_{ \ell \in (\deg(v))_{v \in V}+ 2\Z_{\geq-1}^V} \subset \Z,
 \\& \label{def:Gseq} (G^{\nu}_{\omega,r,\ell})_{\ell \in (\deg(v))_{v\in V}+2\Z^V_{\geq -1}} \subset \bbC,
 \end{align} as follows. The sequence \eqref{def:Fseq} is implicitly  defined by the Laurent expansion of $F_{\Gamma}$ with respect to the variable $(x_v^{\nu_v})_{v \in V}$ given in \eqref{eq:Fexpansion}, which is valid when $\min \{ \abs{x_v^{\nu_v}}, v \in V \}$ is smaller than unity, and the sequence \eqref{def:Gseq} is defined in terms of the sequence \eqref{def:Fseq} by \eqref{def:G_l}
\begin{align} \label{eq:Fexpansion}
F_{\Gamma}(x)=&\sum_{\ell \in (\deg(v))_{v\in V}+2\Z^{V}_{\geq -1}} F^{\nu}_{\ell}  x^{\el},
\\ \label{def:G_l}
G^{\nu}_{\omega,r,\ell}\coloneqq&\lambda_{\nu} \sum_{\alpha \in \frac{1}{2}(\tilde{\omega} + r e) + \Z^V / r \Z^V}
\bm{e}\left( \frac{-Q(\alpha)}{r} \right) \bm{e}\left(\frac{\alpha  \el}{r}\right)  F_{\ell}.
\end{align} The expansion \eqref{eq:Fexpansion} is obtained by taking, for each $v \in V$, a suitable power of the geometric series. In particular, we remark that $\abs{F^{\nu}_{\ell}}$ grows polynomially in $\ell$. Moreover, if we write $F_{\ell}=F^{e}_{\ell},$ where $e\in \mu_2^V$ is the neutal element, then we have for all $\ell \in (\deg(v))_{v \in V}+2\Z_{\geq -1}^V$ that
\begin{equation}  \label{eq:Fsymmetry} 
F^{\nu}_{\ell}=\left(\prod_{v\in V} \nu_v^{\deg(v)}  \right)F_{\ell}.
\end{equation} 
Finally, for all  $\nu \in \mu_2^V, x\in \bbC^V$ and $\tau \in \mathbb{H}$ we define
\begin{align}
\label{def:varepsilon} &\varepsilon(x,\nu)\coloneqq\min \left\{ \abs{e^{\nu_vx_v}}, v \in V \right\},
	\\& \label{eq:Zhatmu} \widehat{Z}^{\nu}_{\omega,r}(\tau)\coloneqq\sum_{\ell \in (\deg(v))_{v\in V}+2\Z_{\geq -1}^V} G^{\nu}_{\omega,r,\ell}  \bm{e}\left( \frac{\tau Q_{\nu}^{-1}(\ell)}{4}\right).
\end{align}

\begin{prop} \label{prop:GeneratingFunction}  We have the following two convergent expansions,  the first of which is valid when the quantity \eqref{def:varepsilon} is smaller than unity, and the second of which is valid for all $\tau \in \mathbb{H}$
	\begin{align} \label{eq:Gexpansion}
	G_{\omega,r}(x)&=\sum_{\ell \in (\deg(v))_{v \in V}+ 2\Z_{\geq-1}^V} G^{\nu}_{\omega,r,\ell}  \exp\left( \transpose{(\el)} x\right),
	\\ \label{eq:CentralIdentity} 
	\widehat{Z}_r(M,\omega;\tau)
	&=
	2^{-\abs{V}} \sum_{\nu \in \mu_2^V } \widehat{Z}^{\nu}_{\omega,r}(\tau).
	\end{align}
	\end{prop}

%We emphasize that \eqref{eq:GgeneratesZ_r} is well-known and was contained in \cite{Costantino-Gukov-Putrov}. To prove equation \eqref{eq:GgeneratesZhat}, we will use the following version of Gaussian reciprocity
\begin{rem} Our proof of  \cref{prop:GeneratingFunction} is inspired by \cite[Section 3.1]{Costantino-Gukov-Putrov}  and reproduces  some of the  computations from \cite[Section 3.1]{Costantino-Gukov-Putrov}. We strongly emphasize, however, that this involves only the purely algebraic computations  from\cite[Section 3.1]{Costantino-Gukov-Putrov}, which involve no infinite summation or limits, and thus our proof is fully rigorous. 
\end{rem}
The proof of \cref{prop:GeneratingFunction} uses the following version of Gaussian reciprocity.

\begin{prop}[{\cite[Theorem 1]{DT}}] \label{prop:reciprocity}
	Let $ L $ be a lattice of finite rank $ n $ equipped with a non-degenerate symmetric $ \Z $-valued bilinear form $ \sprod{, } $.
	We denote the dual lattice as follows
	\[
	L' \coloneqq \{ y \in L \otimes \R \mid \sprod{x, y} \in \Z \text{ for all } x \in L \}.
	\]
	Let $ 0 < k \in \abs{L'/L} \Z, u \in \frac{1}{k} L $, 
	and $ h \colon L \otimes \R \to L \otimes \R $ be a self-adjoint automorphism such that $ h(L') \subset L' $ and $ \frac{k}{2} \sprod{y, h(y)} \in \Z $ for all $ y \in L' $.
	Let $ \sigma $ be the signature of the symmetric bilinear form $(x,y) \mapsto \sprod{x, h(y)} $.
	Then it holds that
	\begin{equation}
	\sum_{x \in L/kL} \bm{e} \left( \frac{1}{2k} \sprod{x, h(x)} + \sprod{x, u} \right)
	=\frac{\bm{e}(\sigma/8) k^{n/2}}{\sqrt{\abs{L'/L} \abs{\det h}}}
	\sum_{y \in L'/h(L')} \bm{e} \left( -\frac{k}{2} \sprod{y + u, h^{-1}(y + u)} \right).
	\end{equation}
\end{prop}

\begin{proof}[Proof of $ \cref{prop:GeneratingFunction} $]
First notice that for all $\nu \in \mu_2^V$ we have that
\begin{align} 
\label{def:G} G_{\omega,r}(x) =&\lambda_{\nu} \sum_{\alpha \in \frac{1}{2}(\tilde{\omega} + r e) + \Z^V / r \Z^V}
\exp\left( \frac{-2\pi i Q(\alpha)}{r} \right)\prod_{v \in V} F_v\left(\exp\left(\nu_v\left(\frac{2\pi i\alpha_v}{r} +x_v\right)\right)\right).
\end{align}	Using this, one readily finds that equation \eqref{eq:Gexpansion} holds.

	 We now turn to  \eqref{eq:CentralIdentity}. Let $\fraks \in (\deg(v))_{v\in V}+ 2\Z^V /2B (\Z^V)$. Write $\widehat{Z}(M,\fraks;q)=\widehat{Z}_{\fraks}(q)$. Recall the notation $Q$ for the quadratic form associated with $-B$. Substituting the expansion \eqref{eq:Fexpansion} into the definition of $\widehat{Z}_{\fraks}(q)$ and using the well-known fact that, when $y=x^{-1}$, we have that $$\oint_{\abs{z}=1-\varepsilon} \frac{d z}{2\pi i z}=\oint_{\abs{y}=1-\varepsilon} \frac{d y}{2\pi i y},$$
	   where it is understood that both the circle in the $\bbC_z$-plane and the circle in the $\bbC_y$-plane are oriented counter-clockwise, we obtain the following identities (where $\delta>0$ is a small positive parameter)
	\begin{align} \begin{split} \label{eq:computation0}
	\widehat{Z}_{\fraks}(q)&=2^{-\abs{V}}q^{\Delta}\sum_{\nu\in \mu_2^V} \prod_{v\in V} \oint_{\abs{x_v^{\nu_v}}=1-\delta} \frac{dx_v}{2\pi i x_v} \sum_{\ell \in \fraks+2B(\Z^V)}q^{Q^{-1}(\ell)/4} x^{\ell} \sum_{w \in (\deg(v))_{v\in V}+2\Z^{V}_{\geq -1}} F^{\nu}_{w} x^{I_{\nu}(w)}
\\&=2^{-\abs{V}}q^{\Delta}\sum_{\substack{\nu \in \mu_2^V,\\w \in (\deg(v))_{v\in V}+2\Z^{V}_{\geq -1}: -I_{\nu}(w) \in \fraks+2B(\Z^V)}} F^{\nu}_{w}  q^{Q^{-1}(I_{\nu}(w))/4}
\\&=2^{-\abs{V}}q^{\Delta} \sum_{\substack{\nu \in \mu_2^V,\\w \in (\deg(v))_{v\in V}+2\Z^{V}_{\geq -1}: -I_{\nu}(w) \in \fraks+2B(\Z^V)}} \left( \prod_{v \in V} \nu_v^{\deg(v)} \right)  F_{w}  q^{Q^{-1}_{\nu}(w)/4},
\end{split} 
	\end{align}
	and where, in the second equality, we used that $Q^{-1}$ is a quadratic form, and, in the third equality we used the identity \eqref{eq:Fsymmetry} and the definition of $Q_{\nu}^{-1}.$ For each $\nu \in \mu_2^V$, define $c_{\nu}\coloneqq\prod_{v \in V} \nu_v^{\deg(v)}.$ By recalling the definition of $\widehat{Z}_r(M,\omega;\tau)$ given in equation \eqref{def:Z_tau} and by writing $q=\bm{e}(\tau+1/r),$ where $\tau \in \mathbb{H}$ is a small paramter, we obtain from \eqref{eq:computation0} the following identity
	\begin{equation} \label{eq:step1}
	\widehat{Z}_r(M,\omega;\tau)=2^{-\abs{V}}\bm{e}\left(\frac{\Delta}{r} \right)\sum_{\substack{\nu \in \mu_2^V,\\s \in \Spin^c(M),\\w \in (\deg(v))_{v\in V}+2\Z^{V}_{\geq -1}:
			\\ -I_{\nu}(w) \in \fraks+2B(\Z^V) }} z_r(\omega,\fraks)  c_{\nu}  F_{w}  \bm{e}\left( \frac{Q^{-1}_{\nu}(w)}{4r}\right)  \bm{e}\left( \frac{\tau Q^{-1}_{\nu}(w)}{4}\right).
	\end{equation} 
	We have a bijection
	\begin{equation}
	(\deg(v))_{v\in V}+2\Z^{V}  \overset{\sim}{\longrightarrow} \bigsqcup_{\fraks \in  (\deg(v))_{v \in V} +2\Z^V /2B (\Z^V)}  \fraks+ 2B(\Z^V),
	\end{equation} 
	and therefore, we always have that $-I_{\omega}(w) \in \fraks+2B(\Z^V)$ for some uniquely determined $\fraks=\fraks(w,\nu) \in \Spin^c(M)$.  It follows that we may rewrite  the identity \eqref{eq:step1} as follows
		\begin{equation} \label{eq:step2}
	\widehat{Z}_r(M,\omega;\tau)=2^{-\abs{V}}\bm{e}\left(\frac{\Delta}{r}\right)\sum_{\substack{\nu \in \mu_2^V,\\w \in (\deg(v))_{v\in V}+2\Z^{V}_{\geq -1}}} z_r(\omega,\fraks(w, \nu)) c_{\nu} F_{w}  \bm{e}\left( \frac{Q^{-1}_{\nu}(w)}{4r}\right)   \bm{e}\left( \frac{\tau Q^{-1}_{\nu}(w)}{4}\right).
	\end{equation} 
	Therefore, to prove \eqref{eq:CentralIdentity}, we must show that for each $\nu \in \mu_2^V$ and each $\ell \in (\deg(v))_{v \in V}+2\Z_{\geq -1}^V$, we have that
	\begin{equation} \label{eq:CentralEquation}
	G^{\nu}_{\omega,r,\ell}=\bm{e}\left(\frac{\Delta}{r} \right) z_r(\omega,\fraks(\el)) \bm{e}\left( \frac{Q^{-1}_{\nu}(\ell)}{4r}\right)  c_{\nu} F_{\ell}.
	\end{equation}
	Proving \eqref{eq:CentralEquation} will be done by an application of Gaussian reciprocity. We have the following integrality fact, which is needed for the application of Gaussian reciprocity
	\begin{equation} \label{eq:integrality}
	B \tilde{\omega} \in 2\Z^V.
	\end{equation}
	Towards proving \eqref{eq:CentralEquation} we follow the cue from  \cite[Equation (3.4)]{Costantino-Gukov-Putrov} and define
	\begin{align}
	&\calC^{\nu}_\ell \coloneqq\sum_{\alpha \in \frac{1}{2}(\tilde{\omega} + r e) + \Z^V / r \Z^V}
	\bm{e}\left( \frac{-Q(\alpha)}{r} \right) \bm{e}\left(\frac{\transpose{\alpha}  \el}{r}\right),
%	\\\label{eq.line2} &=\sum_{n \in \Z^V/ r\Z^V} \bm{e}\left( \frac{\transpose{n} (2B)n}{2r}+\frac{\transpose{n}(\el+B(\tilde{\omega}+re))}{r} -\frac{Q( \tilde{\omega}+re)}{4r}+\frac{\transpose{(\tilde{\omega}+re)}\el}{2r}\right),
	\end{align}
	 so that $G^{\nu}_{\omega,r,\ell}= \lambda   \calC^{\nu}_{\ell} c_{\nu} F_{\ell},$ where we used $\lambda_{\nu} = \lambda  c_{\mu} $. Comparing with \eqref{eq:CentralEquation}, we see that we must prove that 
	\begin{equation} \label{eq:reduction}
	\lambda  \calC^{\nu}_{\ell}=\bm{e}\left(\frac{\Delta}{r} \right) z_r(\omega,\fraks(\el)) \bm{e}\left( \frac{Q^{-1}_{\nu}(\ell)}{4r}\right).
	\end{equation}

Equation \eqref{eq:reduction} is implicitly proven in \cite{Costantino-Gukov-Putrov}. The proof of equation \eqref{eq:reduction} is completely algebraic, and involves no infinite summation or taking limits, and this part of the computation in \cite{Costantino-Gukov-Putrov} is completely rigorous. For the sake of completeness, we repeat the details here. 

Introduce $u=\tilde{\omega}+re$ and $w=\frac{\el+Bu}{r} \in \R^V$. Define
\begin{equation} \label{def:Ctilde}
\tilde{\calC}^{\nu}_\ell\coloneqq\sum_{m \in \Z^V /2B\Z^V}\bm{e}\left(-r\frac{\transpose{m}B^{-1}m}{4}-r \frac{\transpose{m}B^{-1}w}{2}\right).
\end{equation} The integrality condition \eqref{eq:integrality} implies that the hypothesis necessary for Gaussian reciprocity (Lemma \cref{prop:reciprocity}) holds for the lattice $\Z^V$ with the standard Euclidean inner product, the self-adjoint automorphism associated with $2B$, and the vector $w$. Using this, we obtain
	\begin{align}
	&\calC^{\nu}_\ell =\sum_{n \in \Z^V / r \Z^V} \bm{e}\left(\frac{-Q(n+2^{-1}u)+\transpose{(n+2^{-1}u}) \el}{r}\right)
	\\&=\bm{e}\left( -\frac{Q(u)}{4r}+\frac{\transpose{u} \el}{2r}\right)\sum_{n \in \Z^V / r \Z^V} \bm{e}\left(\frac{\transpose{n}(2B)n}{2r}+\frac{\transpose{n}(\el+Bu)}{r}\right)
	\\ \label{eq:applicationGaussianR} 
	&=\bm{e}\left( -\frac{Q(u)}{4r}+\frac{\transpose{u} \el}{2r}\right)\frac{\bm{e}(- \abs{V}/8) (r/2)^{\abs{V}/2}}{\sqrt{\abs{\det B}}}\sum_{m \in \Z^V /2B\Z^V} \bm{e}\left(-\frac{r}{4}\transpose{(m+w)}B^{-1}(m+w)\right)
	\\&=\bm{e}\left( -\frac{Q(u)}{4r}+\frac{\transpose{u} \el}{2r}\right)\frac{\bm{e}(- \abs{V}/8) (r/2)^{\abs{V}/2}}{\sqrt{\abs{\det B}}}\sum_{m \in \Z^V /2B\Z^V} \bm{e}\left(-r\frac{ \transpose{m}B^{-1}m}{4}-\frac{\transpose{\el}B^{-1}\el}{4r}+\frac{Q(u)}{4r}\right)
	\\ & \times \bm{e}\left(\frac{-r}{2}\left(\transpose{m}B^{-1}\el r^{-1}+\transpose{m} u r^{-1}+\transpose{\el} ur^{-2}\right) \right)
	\\&= \bm{e}\left(\frac{Q^{-1}_{\nu}(\ell)}{4r}\right)\frac{\bm{e}(- \abs{V}/8) (r/2)^{\abs{V}/2}}{\sqrt{\abs{\det B}}}\sum_{m \in \Z^V /2B\Z^V}\bm{e}\left(-r\frac{\transpose{m}B^{-1}m}{4}- \frac{\transpose{m}B^{-1}(\el+Bu)}{2}\right)
	\\&=	\bm{e} \left( \frac{1}{4r} Q_{\nu}^{-1}(\ell)\right)
	\frac{\bm{e}(- \abs{V}/8) (r/2)^{\abs{V}/2}}{\sqrt{\abs{\det B}}}
	\tilde{\calC}^{\nu}_\ell,
	\end{align}
	where we applied \cref{prop:reciprocity} in the third equality. By fixing a set of representatives for $\Z^V/ B\Z^V$ we get a well-defined map $(\Z^V/2 \Z^V)\times (\Z^V /B\Z^V) \rightarrow \Z^V/(2B)\Z^V$ given by $(A,a) \mapsto BA+a \mod 2B\Z^V$. This map is easily seen to be injective, and since the source and target are of the same finite cardinality, it is therefore bijective. Thus we can and will write $m=a+BA, a \in \Z^V /2 \Z^V, A \in \Z^V /B \Z^V$ for unique classes. Further, we can define $ s=s(\el) \in (\Z/2\Z)^V $ and $ b=b(\el) \in \Z^V / B \Z^V $ given by the projections of $ \el \in \Z^V $ in $ \Z^V / 2B\Z^V $ by $ \sigma \colon H_1(M, \Z) \times \Spin(M) \xrightarrow{\sim} \Spin^c(M) $, such that 
	\begin{equation} \label{def:sandb}
	2b(\el) + B s(\el) \equiv \el \bmod 2B \Z^V.
	\end{equation} Substituting $m=BA+a$ and \eqref{def:sandb} into $\tilde{\calC}^{\nu}_\ell$ as defined in equation \eqref{def:Ctilde}, and using the integrality condition $B\tilde{\omega} \in 2\Z^V$ (stated in\eqref{eq:integrality}) and the facts that $re =-re \mod 2\Z^V$ and $\transpose{A}b\in \Z^V$, one obtains the following identity (where se used the notation $s=s(\el)$ and $b=b(\el)$ for notational clarity)
	\begin{equation} \label{eq:intermediate}
\tilde{\calC}^{\nu}_\ell =
\sum_{\substack{a \in \Z^V / B \Z^V, \\ A \in \Z^V/2\Z^V}}
	\bm{e}\left( 
		-\frac{\transpose{a} B^{-1} a + 2 \transpose{A} a + \transpose{A} B A}{4r} 
		- \frac{\transpose{a} (2B^{-1} b +s+\tilde{\omega}-re) + \transpose{A} B (s-re)}{2}  
	\right).
	\end{equation}
To further simplify the right hand side of \eqref{eq:intermediate} and to prove equation \eqref{eq:reduction}, one considers separately the cases of $r \equiv 1 \mod 4, r \equiv 2 \mod 4$ and $r \equiv 3 \mod 4$. We recall some identifications and formulas, which are needed for computing the coefficients $z_r(\omega,\fraks)$. We have 
$$\Spin(M) \cong \left\{ s \in (\Z /2 \Z)^V \mid Bs \equiv (B_{vv})_{v \in V} \bmod 2 \Z^V \right\},$$ and we can write the canonical map 
%$ i \colon \Spin(M) \to \Spin^c(M) $ and 
$ \sigma \colon H_1(M, \Z) \times \Spin(M) \xrightarrow{\sim} \Spin^c(M) $, the $\bmod \ 4$ reduction of the Rokhlin invariant of $s\in \Spin(M)$,  the linking pairing $ \lk \colon H_1(M, \Z) \otimes H_1(M, \Z) \to \Q/\Z $ and its quadratic refinement $ q_s \colon H_1(M, \Z) \to \Q/2\Z  $ as follows
\begin{align}  \label{Formulae}\begin{split}  
%i(s) = Bs, \quad
& \sigma(b, s) = 2b + Bs,
\\  &\mu(M, s) = -\abs{V} - \transpose{s}Bs \in \Z / 4 \Z,
\\ & \lk(a, b) = \transpose{a} B^{-1} b \bmod \Z, 
\\& q_s (a) = \transpose{a} B^{-1} a + \transpose{a} s \bmod 2\Z.
\end{split} 
\end{align} By the calculation in \cite[Subsection 3.2--3.4]{Costantino-Gukov-Putrov}, we have
	\begin{align}	
	\tilde{\calC}^{\nu}_\ell =
	\begin{dcases}
	\begin{aligned}
	&\frac{\bm{e}(\abs{V}/8) 2^{\abs{V}/2}}{\sqrt{\abs{\det B}}}
	\bm{e} \left( \frac{1}{4} \transpose{s(\el)} B s(\el) - \frac{1}{4} \tr B + \frac{ \abs{V} - 1}{2} \right)
	\\
	& \sum_{a, f \in \Z^V/B\Z^V}
	\bm{e} \left( -\frac{r-1}{4} \transpose{a} B^{-1} a + \transpose{a} B^{-1}(f - b(\el)) - \frac{1}{2} \transpose{a} \tilde{\omega} + \transpose{f} B^{-1} f  \right)
	\end{aligned}
	& \text{ if } r \equiv 1 \bmod 4, \\
	\begin{aligned}
	2^{\abs{V}} \sum_{{a}\in \Z^V /B\Z^{V}}
	\bm{e} \left( -\frac{r}{4} \transpose{a} B^{-1} a - 2 \transpose{a} B^{-1} b(\el) - \frac{1}{2} \transpose{a} (s(\el)+\tilde{\omega}) \right)
	\end{aligned}
	& \text{ if } r \equiv 2 \bmod 4, \\
	\begin{aligned}
	&\frac{\bm{e}( -\abs{V}/8) 2^{\abs{V}/2}}{\sqrt{\abs{\det B}}}
	\bm{e} \left( -\frac{1}{4} \transpose{s(\el)} B s(\el) + \frac{1}{4} \tr B - \frac{ \abs{V} - 1}{2} \right)
	\\
	& \sum_{a, f \in \Z^V/B\Z^V}
	\bm{e} \left( -\frac{r+1}{4} \transpose{a} B^{-1} a - \transpose{a} B^{-1} (f + b(\el)) - \frac{1}{2} \transpose{a} \tilde{\omega} - \transpose{f} B^{-1} f \right)
	\end{aligned}
	& \text{ if } r \equiv 3 \bmod 4.
	\end{dcases}
	\end{align}
	Thus, by using the definition of the constants $z_r(\omega,\fraks)$ given in equation \eqref{coeffsummary} and by using the formulae  for the canonoical pairing, the Rokhlin invariant and the linking pairing and the spin refined quadratic form presented in \eqref{Formulae} above, we obtain the following identity
	\begin{align}	
	\lambda \calC^{\nu}_\ell = \,
	&\frac{\calT (M, [\omega])}{\abs{\det B}}
	\bm{e} \left( - \frac{1}{4r} \left( 3 \abs{V} + \tr B + \transpose{\el}B^{-1} \el\right) \right)
	\\
	&
	\begin{dcases}
	\begin{aligned}
	&\bm{e} \left( \frac{\abs{V} + \transpose{s(\el)} B s(\el)}{4} - \frac{1}{2} \right)
	\\
	& \sum_{a, f \in \Z^V/B\Z^V}
	\bm{e} \left( -\frac{r-1}{4} \transpose{a} B^{-1} a + \transpose{a} B^{-1}(f - b(\el)) - \frac{1}{2} \transpose{a} \tilde{\omega} + \transpose{f} B^{-1} f  \right)
	\end{aligned}
	& \text{ if } r \equiv 1 \bmod 4, \\
	\begin{aligned}
	\sqrt{\abs{\det B}}
	\sum_{{a}\in \Z^V /B\Z^{V}}
	\bm{e} \left( -\frac{r}{4} \transpose{a} B^{-1} a - 2 \transpose{a} B^{-1} b(\el) - \frac{1}{2} \transpose{a} (s(\el)+\tilde{\omega}) \right)
	\end{aligned}
	& \text{ if } r \equiv 2 \bmod 4, \\
	\begin{aligned}
	&\bm{e} \left( -\frac{\abs{V} + \transpose{s(\el)} B s(\ell)}{4} + \frac{1}{2} \right)
	\\
	& \sum_{a, f \in \Z^V/B\Z^V}
	\bm{e} \left( -\frac{r+1}{4} \transpose{a} B^{-1} a - \transpose{a} B^{-1} (f + b(\el)) - \frac{1}{2} \transpose{a} \tilde{\omega} - \transpose{f} B^{-1} f \right)
	\end{aligned}
	& \text{ if } r \equiv 3 \bmod 4
	\end{dcases}
	\\
	= \,
	&\frac{\calT (M, [\omega])}{\abs{H_1(M, \Z)}}
	\bm{e} \left( - \frac{1}{4r} \left( 3 \abs{V} + \tr B + \transpose{\el}B^{-1} \el \right) \right)
	\\
	&
	\begin{dcases}
	\begin{aligned}
	&-\bm{e}(-\mu(M, s(\el))/4)
	\\
	&
	\sum_{a, f \in H_1(M, \Z)}
	\bm{e} \left( -\frac{r-1}{4}\lk(a,a) +\lk(a,f-b) - \frac{1}{2}\omega(a) +\lk(f,f) \right)
	\end{aligned}
	&\text{ if } r=1\bmod 4,
	\\
	\sqrt{|H_1(M;\Z)|} \sum_{a \in H_1(M, \Z)}
	\bm{e} \left( -\frac{r}{4} q_s(a) - \lk(a,b) - \frac{1}{2} \omega(a) \right),
	&\text{ if } r=2\bmod 4,
	\\
	\begin{aligned}
	&-\bm{e}(\mu(M, s(\el))/4)
	\\
	&
	\sum_{a, f \in H_1(M, \Z)}
	\bm{e} \left( -\frac{r+1}{4}\lk(a,a) -\lk(a,f+b)-\frac{1}{2}\omega(a) -\lk(f,f) \right)
	\end{aligned}
	&\text{ if } r=3\bmod 4,
	\end{dcases}
	\\
	= \,
	&z_r(\omega,\sigma(b(\el), s(\el)))
	\bm{e} \left( - \frac{1}{4r} \left( 3 \abs{V} + \tr B + \transpose{\el}B^{-1} \el\right) \right)
	\\
	= \,
	&z_r(\omega, \fraks(\el))
	\bm{e} \left(\frac{\Delta}{r} + \frac{Q^{-1}_{\nu}(\ell)}{4r} \right).
	\end{align}
Thus we obtained \eqref{eq:reduction}, and this concludes the proof.
	% where $ \fraks(\el) \in (B_{vv})_{v \in V} + 2\Z^V / 2B \Z^V \cong \Spin^c(M) $ is the projection of $ \el\in (\deg v)_{v \in V} + 2 \Z^{V}$.
\end{proof}

% --------------------------------------------------------------------------

%\subsection{Proof of \cref{thm:main,thm:integral}}

% --------------------------------------------------------------------------

%\begin{proof}[Proof of $ \cref{thm:main} $]	Since the map $ \Z_{\ge 0}^V + \delta \to \bbC; \ell \mapsto G^{\nu}_{\omega,r,\ell} $ can be regarded as an element of	\[	\delta(y \in \Z_{\ge 0}^V + \delta)  \bbC[ y_v, \delta(y_v \in a + k \Z) \mid v \in V, a, k \in \Z ],	\]	we have the asymptotic expansion	\begin{align}		\widehat{Z}_r ( M,\omega; \iu t)		&\sim		2^{-\abs{V}} \sum_{\nu \in \{ \pm 1 \}^V }		G_{\omega, r} ( I_\nu x) \odot \bm{e}\left( \frac{Q( I_\nu x)}{4}\right) (\sqrt{t})		\quad \text{as } t \to +0 \\		&\sim		G_{\omega, r} \odot \left( \bm{e}\left( \frac{Q(x)}{4}\right) \right) (\sqrt{t})		\quad \text{as } t \to +0			\end{align}	by \cref{prop:asymp_lim,prop:GeneratingFunction}.	Since $ G_{\omega, r} (x) $ is holomorphic at $ x = 0 $, the limit $ \lim_{t \to +0} \widehat{Z}_r ( M,\omega; \iu t) $ converges and equals to	$ G_{\omega,r}(0) $, which equals to $ Z_r(M,\omega) $ by \cref{prop:G(1)=Z}.\end{proof}

\subsection{Proof of The Integral Representation} \label{sec:integral} We now turn to the proof of \cref{thm:integral}. The proof of this theorem is based on \cref{prop:GeneratingFunction}, the following well-known exact formula for Gaussian integrals recalled in the lemma below, and the lemma of stationary phase approximation (\cref{lem:StationaryPhase}).
\begin{lem} Let $B'\in M_{m\times m}(\bbC)$ be a symmetric and non-degenerate $m\times m$ matrix with positive definite imaginary part, and let $w \in \bbC^{m}$. We have that
\begin{equation} \label{eq:Gaussian}
\int_{\R^m} \exp\left(\frac{i}{2} \transpose{x}B'x +i \transpose{w}x\right) d^m x= \sqrt{\frac{(2\pi i)^m}{\det(B')}} \exp\left(-\frac{i}{2} \transpose{w}(B')^{-1}w \right).
\end{equation}
\end{lem} 
\begin{lem}[{\cite[Chapter 7, Lemma 7.7.3]{Hormander83}}] \label{lem:StationaryPhase}
	Let $m$ be a positive integer. Let $B' \in M_{m\times m}(\bbC)$ be a symmetric non-degenerate $m\times m$ matrix with semi-positive definite real part. Let $G\in C^{\infty}(\R^m;\bbC)$ be a smooth function. For each $j \in \{1,...,m\}$ define the differential operator $D_j\coloneqq-i \frac{\partial}{\partial x_j}$, and for every $C=(c_{i,j})_{1\leq i,j \leq m} \in M_{m\times m}(\bbC)$, define the differential operator $\langle C D, D \rangle\coloneqq \sum_{i,j} c_{i,j}D_i D_j$. For every sector $S\subset \mathbb{H}$ the following Poincare asymptotic expansion holds as $\tau \in S $ tends to $0$, provided the integral on the left hand side is convergent
	\begin{equation} \label{eq:stationaryphase}
	\sqrt{\det\left(\frac{-B'}{2\pi^2 \tau i}\right)}
	\int_{\R^m} \exp\left( \frac{\transpose{x}B'x}{2\pi i \tau }\right) G(x) d x \sim \sum_{l=0}^{\infty} (2\pi i \tau)^l   \frac{\langle {B'}^{-1}D, D \rangle^l(G)}{4^l l!}(0).
	\end{equation} 
\end{lem}
We briefly explain how to get \cref{lem:StationaryPhase} from \cite[Chapter 7, Lemma 7.7.3]{Hormander83}. Recall the involution of $\mathbb{H}$ given by $\tau^* \coloneqq -\tau^{-1}$. Then $\abs{\tau}\tau^* \in U(1)$ is a phase factor, which ranges over a compact set as $\tau$ ranges over a sector $S$. Now, to align with the notation of \cite[Chapter 7, Lemma 7.7.3]{Hormander83}, we define a positive paramter $\omega\coloneqq (\pi\abs{\tau})^{-1}$ and we define a non-degenerate symmetric matrix with positive-definite imaginary part $A \coloneqq \abs{\tau}\tau^* B'$.  With this notation, we have that
\begin{equation} \frac{\transpose{x}B'x}{2\pi i \tau }=\frac{i}{\pi\abs{\tau} }\frac{\transpose{x}(\abs{\tau}\tau^* B')x}{2}=i \omega \frac{\transpose{x}Ax}{2},
\end{equation} and one can apply \cite[Chapter 7, Lemma 7.7.3]{Hormander83} directly to obtain the content of \cref{lem:StationaryPhase}. Note that the explicit bound for the error term given \cite[Chapter 7, Lemma 7.7.3]{Hormander83} implies that this can be bounded uniformly when $\abs{\tau}\tau^* \in U(1)$ range over a compact subset of $U(1)$, which is the case when $\tau$ is confined to a sector.
%With notation as in \cref{lem:StationaryPhase}, we define for all non-negative integers $l$ the constants
%\begin{align} \begin{split}  \label{def:Zhatpertubativecoef}\widehat{Z}_{\omega,r,l}\coloneqq&\frac{\langle -B^{-1}D, D \rangle^l(G_{\omega,r})}{4^l l!}(0),\\ \mathcal{B}_{\omega,r,l}\coloneqq&\frac{(-1)^{l+\abs{V}/2}\widehat{Z}_{\omega,r,l}}{\Gamma\left(l+\abs{V}/2\right)}. \end{split}\end{align} 

\begin{proof}[Proof of $ \cref{thm:integral} $]

Recall the notation $\widehat{Z}^{\mu}_{\omega,r}(\tau)$ introduced in \eqref{eq:Zhatmu}. Let $m=\abs{V}$, and identify $\bbC^V \cong \bbC^m$. Consider the meromorphic $(m,0)$-form  $\Omega_{\tau}$ on $\bbC^m$ given by \eqref{eq:integrand} and the $\ell$-dependent holomorphic $(m,0)$-forms $\Omega_{\tau,\ell}$ on $\bbC^m$ given by \eqref{eq:integrandell}
 \begin{align} 
\label{eq:integrand}
&\Omega_{\tau}(x) \coloneqq \exp\left(\frac{Q(x)}{2\pi i \tau} \right) G_{\omega,r}(ix) d^m x,
\\& \label{eq:integrandell} \Omega^{\nu}_{\tau,\ell}(x) \coloneqq G^{\nu}_{\omega,r,\ell}\exp\left(\frac{Q(x)}{2\pi i \tau}  +i \transpose{\el}x\right) d^m x.
\end{align} Set $c_{\tau}\coloneqq \sqrt{\frac{\det(B)}{(2\pi^2 \tau i)^{m}}}$. We will start by proving that
	\begin{equation} \label{eq:identitymu}
	\widehat{Z}^{\nu}_{\omega,r}(\tau)
	= 
	c_{\tau}
	\int_{\Upsilon_{\nu}} \Omega_{\tau}.
	\end{equation}
	 By\eqref{eq:CentralIdentity}, this will imply that \eqref{eq:integraprep} holds. Taking $B'=B_{\tau}\coloneqq\frac{1}{\pi \tau} B$, and applying \eqref{eq:Gaussian} to
	\eqref{eq:CentralIdentity} we obtain
	\begin{align} \begin{split} \label{integralexpression1} 
	\widehat{Z}^{\nu}_{\omega,r}(\tau)
	&= \sum_{\ell} G^{\nu}_{\omega,r,\ell}\exp\left(-2\pi i \tau \frac{\transpose{\el}B^{-1}\el}{4} \right) =\sum_{\ell} G^{\nu}_{\omega,r,\ell}\exp\left(- i \frac{\transpose{\el}B_{\tau}^{-1}\el}{2} \right) 
	%\\&=\sqrt{\frac{\det(B_{\tau})}{(2\pi i)^{m}}}\sum_{\ell}G^{\nu}_{\omega,r,\ell}\int_{\R^m} \exp\left(\frac{i}{2} \transpose{x}B_{\tau}x +i \transpose{\el}x\right) d^m x
	\\&= c_{\tau}\sum_{\ell} G^{\nu}_{\omega,r,\ell}\int_{\R^m} \exp\left(\frac{Q(x)}{2\pi i \tau}  +i \transpose{\el}x\right) d^m x= c_{\tau}\sum_{\ell} \int_{\R^m} \Omega^{\nu}_{\tau,\ell}.
	\end{split} 
	\end{align}
Recall that by definition $\Upsilon_{\nu}=i\varepsilon\nu+\R^m,$
where $\varepsilon$ is a small positive parameter. Note that we may choose $\varepsilon$ so small that the quantity \eqref{def:varepsilon} is sufficiently small for all $x \in \Upsilon_{\nu}$, which implies that the following sum converges absolultely and locally uniformly on $\Upsilon_{\nu}$
\begin{equation}
\label{sumell} 
\sum_{\ell} \Omega^{\nu}_{\tau,\ell}(x)= \Omega_{\tau}(x), \quad \forall x \in \Upsilon_{\nu}.
\end{equation}

 We will now apply Stokes theorem to deform $\R^m$ to $\Upsilon_{\nu}$, and on this contour we can obtain \eqref{eq:mu} by interchanging integration and summation and using \eqref{eq:Gexpansion}. We now give the the details. Identify $\bbC^m\cong \R^m \times i \R^m$, and for each positive real number $R>0$, consider the compact oriented manifold with corners
\begin{equation}
A_R\coloneqq B^m_R\times i[0,\varepsilon \nu] \subset \R^{m} \times i \R^{m}, 
\end{equation}
where $B^m_R\subset \R^m$ is the $m$-dimensional ball of radius $R$ and centered at the origin, and $[0,\varepsilon\nu]$ denotes the straight line in $\R^m$ from the origin to $\varepsilon \nu$. Let $A_R^{\leq 1} \subset A_R$  denote the complement of the corner points. Denote by $S_R^{m-1}\coloneqq\partial B_R^m$ the sphere which bounds $B_R^m$. This is a holomorphic $(m,0)$-form on $\bbC^m$,  so in particular it is closed, i.e.~$d\Omega^{\nu}_{\tau,\ell}=0.$ By Stokes theorem for oriented manifolds with corners we have that $\int_{\partial (A_R^{\leq 1})} \Omega^{\nu}_{\tau,\ell}=\int_{A_R}  d \Omega^{\nu}_{\tau,\ell}=\int_{A_R} 0=0$, and as the boundary of $A_R^{\leq 1}$ decompose as follows (with the Stokes orientation) \begin{equation}
\partial (A_R^{\leq 1})= \text{Int}(B_R^m)\cup (\text{Int}(B_R^m)+i\varepsilon\nu)\cup \left( S_R^{m-1}\times i(0,\varepsilon\nu) \right),
\end{equation}
we obtain the identity
\begin{align} \begin{split} &\label{vanish} \int_{\R^m} \Omega^{\nu}_{\tau,\ell}= \lim_{R \rightarrow \infty} \int_{B_R^m} \Omega^{\nu}_{\tau,\ell}(x) =  \lim_{R \rightarrow \infty} \left[ \int_{i\varepsilon\nu+B_R^m} \Omega^{\nu}_{\tau,\ell}(x)  +\int_{S^{m-1}_R \times i [0,\varepsilon\nu_j]} \Omega^{\nu}_{\tau,\ell}(x)  \right]
	\\& =\lim_{R \rightarrow \infty} \int_{i\varepsilon\nu+B_R^m} \Omega^{\nu}_{\tau,\ell}(x) =\int_{\Upsilon_{\nu}} \Omega^{\nu}_{\tau,\ell}, \end{split}
	\end{align}
In the second to last equation of \eqref{vanish}, we used the obvious fact that
\begin{equation} \lim_{R \rightarrow \infty} \int_{S_R^{m-1} \times i[0,\varepsilon\nu]} \Omega^{\nu}_{\tau,\ell}(x)=0,
\end{equation}
which easily is seen to follow from the fact that
\begin{equation}
\text{Vol}(S_R^{m-1} \times i[0,\varepsilon\nu])  \sup \left\{ \abs{\Omega^{\nu}_{\tau,\ell}(x)}: x \in S_R^{m-1} \times i[0,\varepsilon\nu] \right\} =\mathcal{O}\left(R^{m-1}e^{\frac{R^2}{\Re(2\pi i \tau)} \inf\left\{ Q(u') : u' \in S_1^{m-1} \right\} }\right),
\end{equation}
and the right hand side of this is exponentially decaying. On the contour $\Upsilon_{\nu}$ the quantity $\varepsilon(x,\nu)$ defined in \eqref{def:varepsilon} is sufficiently small, and therefore the expansion \eqref{eq:Gexpansion} holds. As the function $x \mapsto \exp\left(\frac{Q(x)}{2\pi i \tau}\right)G_{\omega,r}(ix)$ is absolutely integrable over $\Upsilon_{\nu}$, we may therefore apply standard theorems of analysis to interchange summation and integration in order to obtain from \eqref{vanish}, \eqref{integralexpression1} and \eqref{sumell}  the desired identity
\begin{align} \label{eq:mu}
\widehat{Z}^{\nu}_{\omega,r}(\tau)
		= c_{\tau}\sum_{\ell}\int_{\Upsilon_{\nu}}  \Omega^{\nu}_{\tau,\ell}=c_{\tau}\int_{\Upsilon_{\nu}} \sum_{\ell}\Omega^{\nu}_{\tau,\ell}=c_{\tau}\int_{\Upsilon_{\nu}} \Omega_{\tau}.
	\end{align} 
Thus we have established \eqref{eq:identitymu}.

We now turn to the proof of the asymptotic expansion \eqref{eq:expansion}, which uses the method of steepest descent \cite{Fedoryuk}. Since $\Omega_{\tau}$ is meromorphic,  it is in particular a closed holomorphic $(m,0)$-form on the complement of its pole divisor $\mathcal{P}_{\omega,r}$, which is identified in equation \eqref{def:P}. We define an allowable deformation to be a smooth one-parameter family of $m$-dimensional contours $\Upsilon_{\nu}(t), t\in [0,1]$, with $\Upsilon_{\nu}(0)=\Upsilon_{\nu}$ and $\Upsilon_{\nu}'\coloneqq\Upsilon_{\nu}(1)$,  which satisfies for all $t \in [0,1]$ that $\Upsilon_{\nu}(t) \subset \bbC^m \setminus \mathcal{P}_{\omega,r}, $ and $\Upsilon_{\nu}(t) \setminus K= \Upsilon_{\nu} \setminus K $  for some fixed compact set $K \subset \bbC^m$. Such a family determines an immersed manifold $\Gamma_{\nu} \subset \bbC^{m} \setminus\mathcal{P}_{\omega,r}$ of dimension $(m+1)$ such that the pullback of $\Omega_{\tau}$ to $\Gamma_{\nu}$ is integrable and the boundary decomposes as follows with the Stokes orientation $\partial \Gamma_{\nu}=\Upsilon_{\nu} \sqcup (-\Upsilon'_{\nu})$. As $\Omega_{\tau}$ is holomorphic on $\Gamma_{\nu}$ we get by Stokes Theorem that $0=\int_{\Gamma_{\nu}} d\Omega_{\tau}=\int_{\partial \Gamma_{\nu}} \Omega_{\tau} $, and therefore we have that
\begin{equation} \label{eq:mu2}
\int_{\Upsilon_{\nu}} \Omega_{\tau} =\int_{\Upsilon'_{\nu}}  \Omega_{\tau}.
\end{equation} We now choose a suitable allowable deformation. Towards that end, let $\varepsilon'>0$ be a small positive parameter  choosen such that
\begin{equation}
\left\{ x \in \bbC^m: \Re(x), \Im(x) \in  [-\varepsilon',\varepsilon']^m \right\} \cap \mathcal{P}_{\omega,r}=\emptyset.
\end{equation} 
Clearly such $\varepsilon'>0$ exists since $G_{\omega,r}$ is regular at the origin.  We can and will apply an allowable deformation of $\Upsilon_{\nu}$ to obtain a new contour $\Upsilon_{\nu}'$, such that $\Upsilon_{\nu}' \cap \R^m= [-\varepsilon',\varepsilon']^m$, and such that the integrand \eqref{eq:integrand} is exponentially decaying on $\Upsilon_{\nu}'  \setminus \R^m$. In fact,  we can and will choose the deformation $\Upsilon_{\nu}' $ such that for each $j \in \{1,...,m\}$, the projection $\pi_j(\Upsilon_{\nu}')$ of $\Upsilon_{\nu}'$ onto $\bbC_j$ is equal to the following union of line segments
\begin{equation}
\pi_j(\Upsilon_{\nu}' )=\left( i\nu_j\varepsilon+(-\infty,-\varepsilon'] \right) \cup \left( i[0,\nu_j \varepsilon]- \varepsilon' \right) \cup \left( [-\varepsilon',\varepsilon'] \right) \cup \left( i[0,\nu_j \varepsilon]+ \varepsilon' \right)  \cup \left( i\nu_j\varepsilon+[\varepsilon',\infty) \right),
\end{equation}
where $[0,\nu_j \varepsilon]$ denotes the straight line between $0$ and $\nu_j \varepsilon$, and where $\varepsilon>0$ is the small positive parameter entering in the definition of $\Upsilon_{\nu}$. We now argue why the integrand is exponentially decaying on $\Upsilon_{\nu}' \setminus \R^m$. Towards that end let $\{\varphi_0\}\cup\{\varphi_i\}_{i}$ be a smooth partition of unity on $\Upsilon_{\nu}'$, such that $\varphi_0(0)=1$ and $\overline{\text{supp}(\varphi_0)} \subset [-\varepsilon',\varepsilon']^m$. Define \begin{equation}
 C\coloneqq\inf\left\{ \Re(Q(x)): x \in \overline{\Upsilon_{\nu}' \setminus \text{supp}(\varphi_0)}  \right\}.
  \end{equation} Note that 
  \begin{equation} \label{eq:Re}
  \Re(Q(x))=Q(\Re(x))-Q(\Im(x)).
  \end{equation} Provided $\varepsilon$ is neglible compared to $\varepsilon'$, the imaginary part of $x \in \overline{\Upsilon_{\nu}' \setminus \text{supp}(\varphi_0)} $ is always neglible compared to its real part, and therefore we see from \eqref{eq:Re} that $\Re(Q(x))>0$ and consequently $C>0$. Since $\varepsilon'$ is only required to be small compared to $\min \{ \lfloor \tilde{\omega}_j \rfloor,j=1,...,m\}$ (which is strictly positive by assumption) and since $\varepsilon$ can be choosen to be arbitrarily small, we see that we can indeed arrange that $\varepsilon' >> \varepsilon$, and therefore that $C>0$. Therefore the absolute value of the exponential factor of the integrand \eqref{eq:integrand} is bounded from above  on $\overline{\Upsilon_{\nu}' \setminus \text{supp}(\varphi_0)} $ by $\abs{\exp\left(C/2\pi \iu \tau\right)}$. Thus we get from \eqref{eq:mu} and \eqref{eq:mu2} that
\begin{equation} \label{def:Itau} 
	\widehat{Z}^{\nu}_{\omega,r}(\tau)=c_{\tau} \int_{ \Upsilon_{\nu}'} \Omega_{\tau}=c_{\tau} \left[ \int_{ \R^m }  \Omega_{\tau}(x)\varphi_0(x)+\mathcal{O}\left(e^{C/2\pi \iu \tau}\right)\right].
\end{equation}
The term $\mathcal{O}\left(e^{C/2\pi \iu \tau}\right)$ is exponentially suppressed. Thus we obtain a Poincare asymptotic expansion of $	\widehat{Z}^{\nu}_{\omega,r}(\tau)$ by applying stationary phase approximation (\cref{lem:StationaryPhase}) to the integral \begin{equation} \label{def:Itau}
I(\tau)\coloneqq c_{\tau} \int_{ \R^m }  \exp\left(\frac{Q(x)}{2\pi i \tau} \right) G_{\omega,r}(ix)\varphi_0(x) d^m x.
\end{equation}
Doing this asymptotic expansion for each $\nu \in \mu_2^V$ gives the asymptotic expansion \eqref{eq:expansion}, where the coefficients of the expansion (w.r.t. the parameter $2\pi i \tau$) are defined by \eqref{def:Zhatpertubativecoef}. This finishes the proof.
\end{proof}

\begin{proof}[Proof of \cref{thm:main}]
This theorem is a direct consequence of the asymptotic expansion \eqref{eq:expansion}.
	\end{proof}

\begin{rem}
	Recall that the pole divisor $\mathcal{P}_{\omega,r} $ of $G(ix)$ is contained within $\R^V$. The integral representation \eqref{eq:integraprep}  can be interpreted as the identity
	\begin{align}
	 \label{df:CauchyPrincipalValue}
	\widehat{Z}_{r}(M,\omega;\tau)  &=  
	\mathrm{v.p.} \sqrt{\frac{\det(B)}{(2\pi^2 i\tau)^{\abs{V}}}}
	\int_{\R^V} \exp\left( \frac{Q(x)}{2\pi i \tau}\right) G_{\omega,r}(ix)dx
	\\& \coloneqq \lim_{\varepsilon \rightarrow 0} 2^{-\abs{V}} \sum_{\nu \in \mu_2^V} 
	\sqrt{\frac{\det(B)}{(2\pi^2 i\tau)^{\abs{V}}}}
	\int_{i\varepsilon\nu+\R^V} \exp\left( \frac{Q(x)}{2\pi i \tau}\right) G_{\omega,r}(ix)dx,
	\end{align} where the right hand side of equation \eqref{df:CauchyPrincipalValue} is the Cauchy principal value. 
	\end{rem} 

%\begin{rem} Consider the case of $\omega \in H^1(M, \bbC/2\Z) \smallsetminus H^1(M, \Z/2\Z)$. Let $\delta_r\coloneqq-2\pi Im(\tilde{\omega})/r$. For each $\nu \in \mu_2^V$ define $\Upsilon_{\delta_r,\nu}\coloneqq i(\delta_r+ \varepsilon \nu) +\R^V$. The first part of the proof of \cref{thm:integral} can easily be generalized to showing	(with obvious notation)	\begin{equation}	\widehat{Z}_{r}(M,\omega;\tau)  =  \sum_{\nu \in \mu_2^V} \left(\frac{\det(B)}{(8\pi^2 i\tau)^{\abs{V}}}\right)^{\frac{1}{2}}\int_{\Upsilon_{\delta_r,\nu}} \exp\left( \frac{Q(x)}{2\pi i \tau}\right) G_{\omega,r}(ix)dx.	\end{equation}	However, the application of the method of steepest descent becomes more complicated. If the norm of $\delta_r$ is negligble compared to $\min \{2\pi\Re(\lfloor\tilde{\omega_v} \rfloor), v \in V \}$, then the proof of the asymptotic expanion \eqref{eq:expansion} can easily be generalized. In the general case however, when performating the deformation onto a steepest descent contour $\Upsilon_{\delta_r,\nu}'$, one picks up a contribution which is a certain sum of residues of the integrand \eqref{eq:integrand}, which are not exponentially decaying for $\tau$ tending to zero (and which a priori do not cancel).	\end{rem} 

We remark that the arguments given in the proof of \cref{cor:Borel} are standard, but are included here for the sake of completeness and for the sake of the non-expert.
\begin{proof}[Proof of \cref{cor:Borel}]

	  Consider the integral $I(\tau)$ defined in equation \eqref{def:Itau}. The asymptotix expansion \eqref{eq:expansion} is obtained by applying stationary phase approximation (\cref{lem:StationaryPhase}) to $I(\tau)$, and therefore we are interested in computing the inverse Laplace transform of $I(\tau)$. Towards that end, we consider the quadratic form $Q \colon \R^m \rightarrow [0,\infty).$ We want to consider $Q$ as a new coordinate on $\R^m$ and towards that end, we now discuss the Gelfand--Leray transform \cite{Leray}.
	  
	   Let $\lambda_1,...,\lambda_m\in (0,\infty)$ be the eigenvalues of $-B$ (counted with multiplicity). As $Q$ is a non-degenerate quadratic form, there exists an orthogonal transformation $L \colon \R^m \rightarrow \R^m$ such that for all $x\in \R^m$ we have that
	  \begin{equation} \label{eq:pullback}
	  L^*(Q)(x)=\sum_{i=1}^m\lambda_i x_i^2.
	  \end{equation}
	  Define a new basis of $\R^m$ by $\tilde{e}_j=L(e_j),j=1,...,m$, and let $y_j,j=1,...,m$ denote the coordinates with respect to this basis. Then equation \eqref{eq:pullback} implies that $Q(y)=\sum_i \lambda_i y_i^2$ and $d^m y= \det(L) d^m x$. Consider the smooth $m-1$ form on $\R^m \setminus \{0\}$ defined by
	\begin{equation} \label{def:GelfandLeray}
	\Omega\coloneqq  (-1)^{m-1}\sqrt{\frac{\det(B)}{\pi^{m}}} \det(L)^{-1}\frac{\sum_{i=1}^m (-1)^{i+1}y_i  \wedge_{j=1,...,m, j\not=i } d y_j}{2Q}.
	\end{equation}
	We want to show that $\Omega\wedge dQ=\sqrt{\frac{\det(B)}{\pi^{m}}} d^m x$, so that $\Omega$ is the (normalized) Gelfand--Leray transform mentioned in the introduction. We have $dQ=\sum_{i=1}^m 2\lambda_i y_i dy_i$, and therefore
	\begin{align} \begin{split}  \label{eq:desired}
		\Omega\wedge dQ&=(-1)^{m-1} dQ \wedge \Omega
	= \sum_{i=1}^m 2\lambda_i y_i dy_i\wedge \sqrt{\frac{\det(B)}{\pi ^{m}}} \det(L)^{-1}\frac{\sum_{i=1}^m (-1)^{i+1}y_i  \wedge_{j=1,...,m, j\not=i } d y_j}{2Q}
		\\&=\sqrt{\frac{\det(B)}{\pi^{m}}} \det(L)^{-1} \frac{(\sum_{i=1}^m \lambda_i y_i^2) d^my}{Q}=\sqrt{\frac{\det(B)}{\pi^{m}}} \det(L)^{-1}  \frac{Q}{Q} d^m y= \sqrt{\frac{\det(B)}{\pi^{m}}} d^m x.
		\end{split} 
	\end{align}
	Thus $\Omega$ has the desired property. On the complement of the origin, we have that $Q$ defines a proper smooth submersion $\R^{m} \setminus \{0\} \rightarrow (0,\infty)$, which is the projection of a fibre bundle with fibres diffeomorphic to the $(m-1)$-dimensional sphere. Setting $Q(x)=z$, and using the identity \eqref{eq:desired}, we have that
	\begin{align} \begin{split}  \label{eq:standard}
	 I(\tau)=&\frac{1}{\sqrt{2\pi i \tau}^m} \int_{\R^V} \exp\left( \frac{Q(x)}{2\pi i\tau }\right)  G_{\omega,r}(ix) \varphi_0(x) \Omega(x) \wedge  d Q(x)
	 \\=&\frac{1}{\sqrt{2\pi i \tau}^m} \int_{\R^V} \exp\left( \frac{z}{2\pi i\tau }\right)  G_{\omega,r}(ix) \varphi_0(x) \Omega(x) \wedge  d z
	 \\ \label{eq:standard}=& \frac{1}{\sqrt{2\pi i \tau}^m} \int_{0} ^{\infty}\exp\left( \frac{z}{2\pi i\tau }\right) \int_{Q^{-1}(z) \cap \R^m} G_{\omega,r}(ix) \varphi_0(x) \Omega(x) \wedge  d z,
	 \end{split} 
	\end{align}
	where we used a standard identity involving the Gelfand--Leray transform \cite{Leray} (see also \cite[Chapter 7]{Arnold88}). In the following, we will use the notation $Q^{-1}(z)=Q^{-1}(z)\cap \R^m$ to avoid cluttering of notation.  For all $z>0$, we have an orientation preserving diffeomorphism $\varphi_z:Q^{-1}(1) \rightarrow Q^{-1}(z)$ given by $x \mapsto \sqrt{z} x$ for all $x \in Q^{-1}(1)$, and from the definition \eqref{def:GelfandLeray}, it follows that
	\begin{equation}
	\varphi_z^*(\Omega_{\mid Q^{-1}(z)})=\frac{\sqrt{z}^m}{z} \Omega_{\mid Q^{-1}(1)}.
	\end{equation}
For small $z$ we have that that $\varphi_0(x)=1$ for all $x \in Q^{-1}(z)$, and therefore
	\begin{align} \begin{split}  \label{rewrite} 
	&\int_{Q^{-1}(z)}G_{\omega,r}(ix)\varphi_0(x) \Omega(x)=	\int_{Q^{-1}(z)}G_{\omega,r}(ix) \Omega(x) \\=&\frac{\sqrt{z}^m}{z} \int_{Q^{-1}(1)} G_{\omega,r}(i\sqrt{z}x) \Omega(x)=\mathcal{B}_r(M,\omega,\tau)(z)=:B(z), \end{split} 
	\end{align}
	where $\mathcal{B}_r(M,\omega,\tau)$ is the function defined in equation \eqref{def:BorelM}. From \eqref{eq:standard} and \eqref{rewrite} we obtain
	\begin{equation}
	I(\tau)=\frac{1}{\sqrt{2\pi i\tau}^m} \int_0^{\infty} \exp\left(\frac{z}{2\pi i \tau}\right) B(z) dz.
	\end{equation} It is standard fact that  there exists a sequence $(b_l)_{l=0}^{\infty} \subset \bbC$ such that $B(z)$ admits a convergent expansion for small $z$ of the form
\begin{equation} \label{Bexp}
B(z)=\frac{\sqrt{z}^m}{z} \sum_{l=0} b_l  z^l,
\end{equation}
and this claim entails that the asymptotic expansion of $I(\tau)$, which follows from applying stationary phase approximation (\cref{lem:StationaryPhase}), can equivalently be obtained by integrating the expansion \eqref{Bexp} against the Laplace kernel (and then mutiplying the resulting series by $\sqrt{2\pi i\tau}^{-m}$), as stated in the second part of \cref{lem:resummation}. Therefore, the second part of \cref{lem:resummation} implies the identity \eqref{eq:integraprep2}.

 To see that an expansion of the form \eqref{Bexp} holds, one can argue as follows. The existence of an expansion of the form \eqref{Bexp} is easily seen to be equivalent to the claim that for all small $s$ in the complex plane, we have that 
 \begin{equation} 
 \label{eq:final} \int_{Q^{-1}(1)} G_{\omega,r}( isx)  \Omega(x) =\int_{Q^{-1}(1)} G_{\omega,r}( -isx)    \Omega(x),	
 \end{equation} because this will imply that only even powers of $s$ appear with non-zero coefficient in the Taylor expansion with respect to $s$ of the function implictly defined in \eqref{eq:final}. Towards showing that claim, consider the antipodal map $A \colon \R^m \rightarrow \R^m$ given by $A(x)=-x$. We observe that $A$ restricts to an isometry $A:Q^{-1}(1) \rightarrow  Q^{-1}(1)$, and this is orientation preserving if and only if $m-1$ is odd. Thus we see	
\begin{multline}
	 \int_{Q^{-1}(1)} G_{\omega,r}(isx)    \Omega(x) =(-1)^m  \int_{Q^{-1}(1)} A^*(G_{\omega,r}(isx)   \Omega(x) )=(-1)^m  \int_{Q^{-1}(1)} G_{\omega,r}(-isx)   A^*(\Omega(x) )\\ \label{secondtolast} = (-1)^{2m} \int_{Q^{-1}(1)} G_{\omega,r}( -isx)   \Omega(x) = \int_{Q^{-1}(1)} G_{\omega,r}( -isx)    \Omega(x),
 	\end{multline} where we used that $A^*(\Omega(x))=(-1)^m \Omega(x)$, which follows directly from \eqref{def:GelfandLeray}. Thus we have established the expansion \eqref{Bexp} (which is standard). This finishes the proof. \end{proof} 
 
 \subsection{A Complimentary Approach} \label{sec:AlternativeProof}

In this section, we use a complimentary approach to prove the following theorem, which of course is a consequence of \cref{thm:main}.
 \begin{thm} \label{thm:it} It holds that
 $ 	\lim_{t \rightarrow +0} \widehat{Z}_r(M,\omega;it)=Z_r(M,\omega).$
 \end{thm}
 Our proof of \cref{thm:it} is based on \cref{prop:GeneratingFunction} and an Euler--Maclaurin type asymptotic expansion given in \cref{prop:asymp_lim} below.
 
 \begin{dfn} 
 	\label{dfn:Hadamard} Let $ N $ be a positive integer. For a formal Laurent series	\[	\varphi(t_1, \dots, t_N) = \sum_{m \in \Z^N} B_m t_1^{m_1} \dots t_N^{m_N} \in \bbC((t_1, \dots, t_N))	\] 	and a $ C^\infty $ function $ f \colon \R^{N} \to \bbC $ such that $ f^{(m)} (0) $ converges for any $ m \in \Z^N $, define their Hadamard product $ \varphi \odot f (t) \in \bbC((t)) $ as  follows	\[	\varphi \odot f (t)	\coloneqq	\sum_{m \in \Z^N} B_m f^{(m)}(0) t^{m_1 + s + m_N}.	\]
 \end{dfn}
 
 \begin{prop}[{\cite[Proposition 3.6]{M_GPPV}}] 
 	\label{prop:asymp_lim} 	Let $ N $ be a positive integer, $ \lambda \in \Z^N $ and $ F \colon \Z_{\ge 0}^N + \lambda \to \bbC $ be an element of 	\[ 	\delta(y \in \Z_{\ge 0}^N + \lambda)  	\bbC[ y_i, \delta(y_i \in a + k \Z) \mid 1 \le i \le N, a, k \in \Z ],	\]	where $ \delta $ be the Kronecker delta function and we stipulate that $ a + k\Z = \{ a \} $ when $ k=0 $.	Let 
 		\[	\varphi_{F, u} (t_1, \dots, t_N)	\coloneqq	\sum_{l \in \Z_{\ge 0}^N + \lambda} F(l) 	e^{t_1 (l_1 + u_1) + s + t_N (l_N + u_N)}	\in \bbC ((t_1, \dots, t_N))[u_1, \dots, u_N].
 			\]	Then, for any $ \alpha \in \R^N $ and any $ C^\infty $ function $ f \colon \R^{N} \to \bbC $ of rapid decay as $ x_1, \dots, x_N \to \infty $,	the following asymptotic formula holds as $t$ tends to $+0$	\[	\sum_{l \in \Z_{\ge 0}^N + \lambda} F(l) f(t(l+\alpha))	\sim	\varphi_{F, \alpha} \odot f (t).
 			\]
 		\end{prop} 
 	\cref{prop:asymp_lim} generalizes \cite[Equation (44)]{Zagier_asymptotic}, \cite[Equation (2.8)]{BKM} and \cite[Lemma 2.2]{BMM_high_depth}.\cref{prop:asymp_lim} was proved by using the Euler--Maclaurin summation formula.

\begin{proof}[Proof of $ \cref{thm:it} $]
	Since the map $ \Z_{\ge 0}^V + \delta \to \bbC; \ell \mapsto G^{\nu}_{\omega,r,\ell} $ can be regarded as an element of
	\[
	\delta(y \in \Z_{\ge 0}^V + \delta) \cdot \bbC[ y_v, \delta(y_v \in a + k \Z) \mid v \in V, a, k \in \Z ],
	\]
	we have by \cref{prop:asymp_lim,prop:GeneratingFunction} the following asymptotic expansion as $t$ tends to $+0$
	\begin{align}
	\widehat{Z}_r ( M,\omega; \iu t)
	&\sim
	2^{-\abs{V}} \sum_{\nu \in \{ \pm 1 \}^V }
	G_{\omega, r} ( I_\nu x) \odot \bm{e}\left( \frac{Q( I_\nu x)}{4}\right) (\sqrt{t})
	\\
	&\sim
	G_{\omega, r} \odot \left( \bm{e}\left( \frac{Q(x)}{4}\right) \right) (\sqrt{t}).
	\end{align}
	Since $ G_{\omega, r} (x) $ is holomorphic at $ x = 0 $, the limit $ \lim_{t \to +0} \widehat{Z}_r ( M,\omega; \iu t) $ converges and equals to
	$ G_{\omega,r}(0) $, which equals to $ Z_r(M,\omega) $ by \cref{prop:G(1)=Z}.
\end{proof}

\begin{rem} \label{rem:comparison} We here give a brief comparison between \cref{thm:it} and \cref{thm:integral}. \cref{thm:it}  admits a shorter proof based on \cref{prop:asymp_lim} and is valid for $\tau \in i\R_+$. \cref{thm:integral} on the other hand is applicable for $\tau$ in any sector $S\subset \mathbb{H}$ and gives more information on the full asymptotic expansion and the Borel transform, but has a longer proof. 
\end{rem}

\bibliographystyle{alpha}
\bibliography{CGP_radial_limit}

% --------------------------------------------------------------------------
\end{document}